\documentclass{amsart}
\usepackage{graphicx, amsmath, amssymb, amsthm, amsfonts, fullpage, latexsym,enumerate, xcolor, commath, tikz, tikz-cd, verbatim, hyperref, mathtools}
\usepackage[hyphenbreaks]{breakurl}
\tikzcdset{scale cd/.style={every label/.append style={scale=#1},
    cells={nodes={scale=#1}}}}
\usepackage[backend=biber,style=alphabetic]{biblatex}
\title{Chow rings of matroids as permutation representations}
\author{Robert Angarone, Anastasia Nathanson, Victor Reiner}
\address{University of Minnesota - Twin Cities, Minneapolis MN 55455}
\email{angar017@umn.edu, natha129@umn.edu, reiner@umn.edu}


\theoremstyle{plain}
\newtheorem{thm}{Theorem}[section]
\newtheorem{conj}[thm]{Conjecture}
\newtheorem{question}[thm]{Question}
\newtheorem{lemma}[thm]{Lemma}
\newtheorem{prop}[thm]{Proposition}
\newtheorem{cor}[thm]{Corollary}

\theoremstyle{definition}
\newtheorem{defn}[thm]{Definition}
\newtheorem{example}[thm]{Example}
\newtheorem{remark}[thm]{Remark}

\newcommand{\xx}{\mathbf{x}}

\newcommand{\R}{\mathbb{R}} 
\newcommand{\Z}{\mathbb{Z}}
\newcommand{\C}{\mathbb{C}}
\newcommand{\kk}{\mathbb{F}}

\newcommand{\M}{\mathcal{M}} 
\newcommand{\symm}{\mathfrak{S}} 
\newcommand{\DD}{\mathfrak{D}} 

\newcommand{\G}{\mathcal{G}} 
\newcommand{\N}{\mathcal{N}} 

\newcommand{\subgroup}{G}

\newcommand{\Hom}{\mathrm{Hom}} 
\newcommand{\Tor}{\mathrm{Tor}} 

\newcommand{\E}{E} 

\newcommand{\supp}{\mathrm{supp}} 
\newcommand{\fy}{\mathrm{FY}}
\newcommand{\FY}{\mathrm{FY}}
\newcommand{\spn}{\mathrm{span}}
\newcommand{\init}{\mathrm{in}}
\newcommand{\into}{\hookrightarrow}

\renewcommand{\phi}{\varphi}
\renewcommand{\L}{\mathcal{L}}

\DeclareMathOperator\rk{\mathrm{rk}}
\DeclareMathOperator\cork{\mathrm{cork}}
\DeclareMathOperator\Aut{\mathrm{Aut}}

\DeclareRobustCommand{\a}[1]{ {\begingroup\color{orange}  {Anastasia: #1}\endgroup} }

\begin{document}

\maketitle

\begin{abstract}
Given a matroid with a symmetry group, we study the induced group action on the Chow ring of the matroid with respect to symmetric building sets. This turns out to always be a permutation action. Work of
Adiprasito, Huh and Katz showed that the Chow ring satisfies Poincar\'e duality and
the Hard Lefschetz theorem.  We lift these to statements about this permutation action,
and suggest further conjectures in this vein.
\end{abstract}

\section{Introduction}
\label{intro-section}
A {\it matroid} $\M$ is a combinatorial abstraction of lists of vectors $v_1,v_2,\ldots,v_n$ in a vector space, recording only the information about which subsets of the vectors
are linearly independent or dependent, forgetting their coordinates-- see
Section~\ref{matroid-background-section} for definitions and references.
In groundbreaking work, Adiprasito, Huh and Katz \cite{AHK} affirmed long-standing conjectures of Rota--Heron--Welsh and Mason about vectors and matroids via a new methodology.  Their work employed
a certain graded $\Z$-algebra $A=\bigoplus_{k=0}^r A^k$
called the \emph{Chow ring} for a matroid $\M$ of rank $r+1$, introduced by Feichtner and Yuzvinsky \cite{FY} as a generalization of the Chow ring of DeConcini and Procesi's \emph{wonderful compactifications} for hyperplane arrangement complements \cite{DCP}. 
A remarkable integral Gr\"obner basis result proven by Feichtner and Yuzvinsky \cite[Thm. 2]{FY} shows that each homogeneous component of $A$ is free abelian:  $A^k \cong \Z^{a_k}$ for a positive integer sequence $(a_0,a_1,\ldots,a_r)$.

A key step in \cite{AHK} shows not only that 
$(a_0,a_1,\ldots,a_r)$ is
{\it symmetric} and {\it unimodal}, that is,
\begin{align}
\label{symmetry}
&a_k = a_{r-k}\text{ for } k \leq r/2\\
\label{unimodality} 
&a_0 \leq a_1 \leq \cdots \leq a_{\lfloor \frac{r}{2} \rfloor} =
a_{\lceil \frac{r}{2} \rceil} \geq \cdots \geq a_{r-1} \geq a_r,
\end{align}
but in fact proves this as a corollary of something much stronger: the Chow ring $A$ enjoys a trio of properties referred to as the {\it K\"ahler package}, reviewed in Section~\ref{AHK-section} below.  The first of these properties is
{\it Poincar\'e duality}, proving \eqref{symmetry} via a natural
$\Z$-module isomorphism $A^{r-k} \cong \Hom_\Z(A^k,\Z)$.  The second property, called the {\it Hard Lefschetz Theorem}, shows that after tensoring $A$ over $\Z$ with $\R$ to obtain $A_\R=\bigoplus_{k=0}A^k_\R$, one can find {\it Lefschetz elements} $\omega$ in $A^1_\R$ such that multiplication by $\omega^{r-2k}$ gives $\R$-linear isomorphisms $A^k_\R \rightarrow A^{r-k}_\R$ for $k \leq \frac{r}{2}$.  In particular, multiplication by $\omega$ maps $A^k_\R \rightarrow A_\R^{k+1}$ {\it injectively} for $k < \frac{r}{2}$, strengthening the unimodality assertion \eqref{unimodality}.

Feichtner and Yuzvinsky defined the Chow ring $A(\L_\M,\G)$ for any choice of a {\it building set} $\G$ inside the {\it lattice of flats} $\L_\M$ for the matroid $\M$ and gave their integral Gr\"obner basis presentation in that context; these notions are reviewed in Section~\ref{background-section} below.
While the results of \cite{AHK} were proven for the {\it maximal building set} $\G=\L_\M \setminus \{\hat{0}\}$, the Chow ring satisfies the K\"ahler package for any building set (and even for Chow rings of {\it polymatroids}), as shown by Pagaria and Pezzoli \cite[Thm.~4.21]{PagariaPezzoli}.


We are interested  in how the {\it Poincar\'e duality} and {\it Hard Lefschetz} properties interact with symmetry.  We consider any subgroup $\subgroup$ of the group $\Aut(\M)$ of symmetries of the matroid $\M$, assuming that the building set $\G$ is also setwise $\subgroup$-stable.  We observe (see Section~\ref{matroid-background-section} below) that in this situation, $\subgroup$ acts via graded $\Z$-algebra automorphisms
on $A(\L_\M, \G)$, giving $\Z \subgroup$-module structures on each $A^k$, and
$\R \subgroup$-module structures on each $A^k_\R$.  One can also check that $A^r \cong \Z$ with trivial $\subgroup$-action, under one additional technical assumption, that $\G$ contains the ground set of
the matroid--see the proof of Corollary~\ref{integral-equivariant-PD-cor} below.  From
this, the Poincar\'e duality pairing immediately gives 
rise to a $\Z \subgroup$-module isomorphism
\begin{equation}
\label{integral-rep-PD-isomorphism}
A^{r-k} \cong \Hom_\Z(A^k,\Z)
\end{equation}
where $g$ in $\subgroup$ acts on $\varphi$ in $\Hom_\Z(A^k ,\Z)$ via $\varphi \mapsto \varphi \circ g^{-1}$; similarly
$A^{r-k} \cong \Hom_\R(A^k,\R)$ as $\R \subgroup$-modules.  
Furthermore, we observe that
one can pick a Lefschetz element $\omega$ 
which is $\subgroup$-fixed (see Corollary~\ref{AHK-equivariant-Hard-Lefschetz} below), giving $\R \subgroup$-module isomorphisms and injections

\begin{align}
    \label{real-Lefschetz-PD-isomorphism}
    A_\R^{k} &\overset{\sim}{\longrightarrow} A_\R^{r-k} 
     \quad \text{ for }k \leq \frac{r}{2} \nonumber \\
    a &\longmapsto  a \cdot \omega^{r-2k}
\end{align}
\begin{align}
  \label{real-Lefschetz-injection}
        A_\R^{k} &\hookrightarrow A_\R^{k+1} \quad \text{ for }k < \frac{r}{2} \nonumber \\
        a & \longmapsto  a \cdot \omega. 
\end{align}

We wish to view the isomorphism \eqref{real-Lefschetz-PD-isomorphism} and injection \eqref{real-Lefschetz-injection}
as lifting the numerical equality \eqref{symmetry} and inequality \eqref{unimodality} from $\Z$ to the {\it ring of virtual (complex) characters} $R_\C(\subgroup)$.  Recall that this ring $R_\C(G)$ is a subring of the ring of (conjugacy) class functions $\{f: \subgroup \rightarrow \C\}$ with pointwise addition and multiplication.  It is defined as the free $\Z$-submodule with basis given by the irreducible complex characters $\{ \chi_1,\ldots,\chi_M\}$, where the number $M$ of irreducible characters coincides with the number of conjugacy classes of $\subgroup$.  Thus every virtual character $\chi$ in $R_\C(\subgroup)$ has a unique expansion $\chi = \sum_{i=1}^M a_i \chi_i$ for some $a_i \in \Z$.  If $a_i \geq 0$ for $i=1,2,\ldots,M$,
call $\chi$ a {\it genuine character}, and write $\chi \geq_{R_\C(\subgroup)} 0$.
Similarly, write $\chi \geq_{R_\C(\subgroup)} \chi'$ when $\chi-\chi' \geq_{R_\C(\subgroup)} 0$.
There is a surjective ring map
\begin{equation}
\label{characters-to-integers-map}
R_\C(\subgroup) \longrightarrow \Z
\end{equation}
sending a virtual character $\chi$ to its value $\chi(e)$ on the identity $e$ of $\subgroup$, carrying genuine characters $\chi \geq_{R_\C(\subgroup)} 0$ to nonnegative integers $\Z_{\geq 0}$.  Through this map, equalities and inequalities in $R_\C(\subgroup)$ lift inequalities in $\Z$.  For example,
\eqref{real-Lefschetz-PD-isomorphism}, \eqref{real-Lefschetz-injection}
give rise to equalities and inequalities in $R_\C(G)$ of this form
\begin{align}
\label{character-symmetry}
\chi_{A_\R^k}&=\chi_{A_\R^{n-k}} \quad \text{
for }k \leq \frac{r}{2},\\
\label{character-unimodality}
\chi_{A_\R^k} &\leq_{R_\C(G)} \chi_{A_\R^{k+1}}
\quad \text{ for }k < \frac{r}{2},
\end{align}
which lift the equalities and inequalities \eqref{symmetry}, \eqref{unimodality} in $\Z$
 through the map \eqref{characters-to-integers-map}.

Our goal in this paper is to use  Feichtner and Yuzvinsky's Gr\"obner basis result, along with some combinatorics of {\it nested sets} (reviewed in Section~\ref{nested-set-subsection}), to prove a combinatorial strengthening/lifting of the isomorphisms and injections \eqref{integral-rep-PD-isomorphism}, \eqref{real-Lefschetz-PD-isomorphism}, \eqref{real-Lefschetz-injection}.  For the sake of stating this, recall (or see Section~\ref{matroid-background-section} below) that a (simple) matroid $\M$ can be specified by
its {\it lattice of flats} $\L_\M$;  in the case where $\M$ is realized by
a list of vectors $v_1,v_2,\ldots,v_n$ in a vector space, a subset $F \subseteq \{1,2,\ldots,n\}=:E$ is a flat when  $\{v_j\}_{j \in F}$ is linearly closed, meaning that every vector $v_i$ for $i\in E$ that lies in the linear span of $\{v_j\}_{j \in F}$ already has $i$ in $F$.
This $\L_\M$ is a ranked lattice, whose rank function $\rk: \L_\M \rightarrow \{0,1,2\ldots\}$ is modeled after the dimension of the span of $\{v_j\}_{j \in F}$. A building set $\G \subseteq \L_\M$ is a subset of $\L_\M$ satisfying axioms that roughly say every flat $F$ in $\L_\M$ can be ``built" in a certain way from elements of $\G$.
The building set $\G$ distinguishes certain subsets $N=\{F_1,\ldots,F_\ell\} \subset \G$ called {\it $\G$-nested sets}.
To each flat $F$ in the $\G$-nested set $N$, we need a crucial quantity
\begin{equation}
\label{crucial-quantity}
m_N(F):=\rk(F) - \rk(\vee N_{<F})
\end{equation}
where $\vee  N_{<F}$ denotes the lattice join in $\L_\M$ of all elements of $N$ strictly
below $F$.
Then the \emph{Chow ring} $A(\L_\M, \G)$ of $\M$ with respect to the building set $\G$
is presented as a quotient of the
polynomial ring $S:=\Z[x_F]$ having one variable $x_F$ for each flat $F$ in $\G$.  The presentation takes the form 
$
A(\L_\M, \G) := S / (I + J)
$
where $I ,J$ are certain ideals of $S$ defined more precisely in Definition~\ref{Chow-ring-definition} below.
Feichtner and Yuzvinsky exhibited (see Theorem~\ref{FY-GB-theorem}, Corollary~\ref{cor: mon_basis} below) a Gr\"obner basis for $I + J$ that leads to the following standard monomial $\Z$-basis for $A(\L_\M, \G)$,
which we will call the {\it FY-monomials} of $\M$:
$$
\fy:=\left\{x_{F_1}^{m_1}  \cdots x_{F_\ell}^{m_\ell} \colon
N:=\{F_1, \cdots,  F_\ell\} \text{ is } \G\text{-nested, and } 
0 \leq m_i < m_N(F_i) \text{ for }i=1,2,\ldots,\ell.\right\}
$$

The subset $\FY^k$ of FY-monomials $x_{F_1}^{m_1} \cdots x_{F_\ell}^{m_\ell}$ of total degree $m_1+\cdots+m_\ell=k$ then gives a $\Z$-basis for $A^k$.
One can readily check
(see Corollary~\ref{Chow-ring-carries-perm-reps-cor}) that the group $\subgroup$ permutes the $\Z$-basis $\FY^k$ for $A^k$, 
endowing $A^k$ with the structure of
a {\it permutation representation}, or {\it $\subgroup$-set}.  
Our main result, proven in Section~\ref{main-theorem-section}, is this strengthening of the isomorphisms and injections \eqref{integral-rep-PD-isomorphism}, \eqref{real-Lefschetz-PD-isomorphism}, \eqref{real-Lefschetz-injection}.

\begin{thm}
\label{main-theorem}
Let $\M$ be a simple matroid rank $r+1$ on ground set $E$.  Let $\subgroup$
be a group automorphisms of $\M$, and $\G$ a building set in
$\L_\M$
that contains $E$, is setwise $G$-stable, and satisfies this
{\it stabilizer condition}\footnote{This condition was missing in {\tt arXiv} version 2 of this paper--  the authors 
thank R. Pagaria for pointing out the issue.}:
\begin{equation}
\label{eq:stabilizer-condition}
  \text{for any }\G\text{-nested set }
   N=\{F_i\}_{i=1,\ldots,\ell}, \text{ if }g \in G\text{ has } 
    g(N)=N, \text{ then }g(F_i)=F_i \text{ for } i=1,\ldots,\ell.
\end{equation}
Then there exist
\begin{itemize}
\item[(i)]
$\subgroup$-equivariant  bijections 
$
\pi: \FY^k  \overset{\sim}{\longrightarrow}  \FY^{r-k}
$ for $k \leq \frac{r}{2}$, and
\item[(ii)]
$\subgroup$-equivariant injections
$
\lambda: \FY^k  \hookrightarrow  \FY^{k+1}
$
for $k < \frac{r}{2}$.
\end{itemize}
\end{thm}

\begin{example} 
        Let $\M=U_{4,5}$ be the uniform matroid of rank $4$ on $E=\{1,2,3,4,5\}$, associated to $5$ {\it generic} vectors $v_1,v_2,v_3,v_4,v_5$ in a $4$-space, so that any quadruple $v_i,v_j,v_k,v_\ell$ is linearly independent.  Choose $\G=\G_{\max}=\L_\M\setminus\{\varnothing\}$, the \emph{maximal} building set.
        Then $\L_\M$ has these flats of various ranks:

\begin{center}
\begin{tabular}{|c|c|}\hline
rank & flats $F \in \L_\M$ \\ \hline\hline
 $0$ & $\varnothing$ \\ \hline
 $1$ & $1,2,3,4,5$ \\ \hline
 $2$ & $12,13,14,15,23,24,25,34,35,45$ \\ \hline
 $3$ & $123,124,125,134,135,145,234,235,245,345$ \\ \hline
 $4$ & $E=12345$ \\ \hline
\end{tabular}
\end{center}
The Chow ring $A(\L_\M, \G)=S/(I+J)$, where 
$
S= \Z[ x_i, x_{jk}, x_{\ell m n}, x_E]
$
with
$\{i\}, \{j,k\}, \{\ell, m, n\}$ running through all one, two and three-element subsets of 
$E=\{1,2,3,4,5\}$,
and 
$$
I=\Big( x_F x_{F'} \Big)_{F \not\subset F', F' \not \subset F},
\qquad
J=\bigg( x_i 
+ \sum_{\substack{1 \leq j < k \leq 5\\i \in \{j,k\}}} x_{jk}
+ \sum_{\substack{1 \leq \ell  < m < n \leq 5\\i \in \{\ell,m,n\} }} x_{\ell m n}
\,\,\, + x_E \bigg)_{i=1,2,3,4,5}.
$$
The FY-monomial bases for $A^0,A^1,A^2,A^3$ are shown here,
together with the $\subgroup$-equivariant maps $\lambda$:
$$
\begin{array}{cccccccc}
\bf{FY^0} & & \bf{FY^1}& & \bf{FY^2} & & &\bf{FY^3} \\
 & & & & & & & \\
1 &\overset{\lambda}{\longmapsto}& x_E &\overset{\lambda}{\longmapsto} & x_E^2 &  & &x_E^3 \\
 & & & & & & &\\
  & & x_{ijk} &\overset{\lambda}{\longmapsto} & x_{ijk}^2& & &\\
      & &  1 \leq i<j<k \leq 5 & & & & & \\
       & & & & & & &\\
  & &x_{ij} &\overset{\lambda}{\longmapsto} &x_{ij} \cdot x_E& & &\\
   & & 1 \leq i<j \leq 5& & & &  &\\
\end{array} 
$$
Therefore in this case, the ranks of the free $\Z$-modules $(A^0,A^1,A^2,A^3)$ form the symmetric, unimodal sequence $(a_0,a_1,a_2,a_3)=(1,21,21,1)$. 
Here the bijection $\pi: \fy^0 \rightarrow \fy^3$ 
necessarily maps $1 \longmapsto x_E^3$, and 
the bijection $\pi: \fy^1 \rightarrow \fy^2$  coincides with
the map $\lambda: \fy^1 \rightarrow \fy^2$ above.  
\end{example}

We wish to also view Theorem~\ref{main-theorem} as lifting \eqref{character-symmetry}, \eqref{character-unimodality} from the virtual character ring $R_\C(G)$ to the 
{\it Burnside ring $B(\subgroup)$ of virtual $G$-sets}.  Recall that to define the Burnside ring $B(G)$
(see, e.g., Bouc \cite{Bouc}), one starts with a free $\Z$-module having as basis the $\subgroup$-equivariant isomorphism classes $[X]$ of finite $\subgroup$-sets $X$.  Then $B(\subgroup)$ is the quotient $\Z$-module that mods out by the span of all elements $[X \sqcup Y] - ([X]+[Y])$.  
Multiplication in $B(\subgroup)$ is induced from the rule $[X] \cdot [Y] = [X \times Y]$.
It turns out that $B(\subgroup)$ has a $\Z$-basis given by the isomorphism classes of $\{ [\subgroup/\subgroup_i] \}_{i=1}^N$ of the transitive $G$-sets $G/G_i$ as $\subgroup_1,\ldots,\subgroup_N$ run through representatives of the $\subgroup$-conjugacy classes of subgroups of $\subgroup$. 
Thus every element $b$ of $B(\subgroup)$ has a unique expansion $b = \sum_{i=1}^N a_i [\subgroup/\subgroup_i]$ for some $a_i \in \Z$.  If $a_i \geq 0$ for $i=1,2,\ldots,N$, then write $b \geq_{B(\subgroup)} 0$ and call $b$ a {\it genuine} element of $B(G)$, 
since it is the class $b=[X]$ 
of the (genuine, not virtual) 
$G$-set $X$ that has a total of  $\sum_{i=1}^N a_i$ disjoint $G$-orbits, among which $a_i$ of the orbits are isomorphic to the transitive $G$-set $G/G_i$.
Similarly, write $b \geq_{B(\subgroup)} b'$ when $b-b' \geq_{B(\subgroup)} 0$.  The map which sends the class $[X]$ of a $\subgroup$-set $X$ to the character $\chi_{\C[X]}(g)=\#\{x \in X:g(x)=x\}$ of its $\subgroup$-permutation representation $\C[X]$ gives a natural ring map
\begin{equation}
    \label{Burnside-to-character-map}
B(\subgroup) \longrightarrow R_\C(\subgroup)
\end{equation} 
This map sends genuine elements $b \geq_{B(\subgroup)} 0$ in $B(\subgroup)$ to genuine characters $\chi \geq_{R_\C(\subgroup)} 0$ in $R_\C(G)$.  In this way, equalities and inequalities in $B(\subgroup)$ lift equalities and inequalities in $R_\C(\subgroup)$.
For example, Theorem~\ref{main-theorem} (i),(ii) 
lead to equalities and inequalities in $B(G)$ of this form
\begin{align}
\label{Burnside-symmetry}
[\FY^k]&=[\FY^{n-k}] 
\quad \text{ for }k \leq \frac{r}{2},\\
\label{Burnside-unimodality}
[\FY^k] &\leq_{B(G)} [\FY^{k+1}]
\quad \text{ for }k < \frac{r}{2},
\end{align}
which lift the equalities and inequalities \eqref{character-symmetry}, \eqref{character-unimodality} in $R_\C(G)$
 through the map \eqref{Burnside-to-character-map}.

Before proving Theorem~\ref{main-theorem} in Section~\ref{main-theorem-section}, the background Section~\ref{background-section} reviews properties of building sets of atomic lattices, Chow rings, and matroids. It also collects a few simple observations,
including one new and crucial numerical fact about $\G$-nested sets, Lemma~\ref{lem:crucial-numerical-fact}.
Section~\ref{conjectures-section} pursues the theme of lifting inequalities in $\Z$ to inequalities $R_\C(G)$ and in $B(G)$, going beyond unimodality to conjectures that would extend recent log-concavity and total positivity conjectures for Hilbert functions of Chow rings $A(\L_\M, \G)$.
Section~\ref{sec: further-questions} poses additional questions
and conjectures.

\section*{Acknowledgements}
The authors thank Federico Ardila, Alessio D'Ali, Graham Denham, Chris Eur, Eva-Maria Feichtner, Luis Ferroni, June Huh, Matt Larson, Hsin-Chieh Liao,  Diane Maclagan, Matt Maestroni, Jacob Matherne, Ethan Partida, Connor Simpson, Lorenzo Vecchi and Peter Webb for helpful conversations and references.  They particularly thank Roberto Pagaria for
alerting them to an issue in the previous statement of Theorem~\ref{main-theorem} that omitted the stabilizer condition \eqref{eq:stabilizer-condition}. 
They thank an anonymous referee for helpful suggestions that improved the exposition.  Additionally, they are grateful to Trevor Karn for his wonderful {\tt Sage/Cocalc} code that checks whether a symmetric group representation is a permutation representation.  The authors received partial support from NSF grants 
DMS-1949896, DMS-1745638, DMS-2053288, DMS-2230648. The second author received partial support from PROM project no. POWR.03.03.00-00-PN13/18.


\section{Background}
\label{background-section}

\subsection{Lattices and Building Sets} 
\label{building-set-basics}
Although our results concern matroids, their lattices
of flats, building sets, and Chow rings, Feichtner and Yuzvinsky \cite{FY} actually worked at a higher level of generality: building sets in {\it atomic lattices} and their
Chow rings.  We start with this more general framework,
in order to present their integral Gr\"obner basis result.

Throughout, $\L$ will be a finite lattice, that is a finite poset in which any pair of elements $x,y$ has 
a least upper bound and greatest lower bound, called their join $x \vee y$ and meet $x \wedge y$.  Finiteness then implies than any subset $A \subseteq \L$ also has such 
a join $\vee A$ and meet $\wedge A$;  in particular,
the entire lattice
$\L$ itself has a top element $\hat{1}=\vee \L$ 
and bottom element $\hat{0}=\wedge \L$.  
For elements $F,F'$ in $\L$, the (closed) interval between them will be denoted $[F,F']:=\{ F'' \in \L: F \leq F'' \leq F'\}$.

\begin{defn} \cite[Def.~1]{FY}
    In a finite lattice $\L$, a subset $\G \subseteq \L \setminus \{\hat{0}\}$ is a \emph{building set} if for every $F$ in $\L \setminus \{\hat{0}\}$, letting
    $\max(\G_{\leq F})=\{F_1,\ldots,F_\ell\}$ denote the maximal elements of $\G_{\leq F}:=\{F' \in \G: F' \leq F\}$, one has  a poset isomorphism
    \[
        [\hat{0}, F] 
        \cong 
         [\hat{0},F_1] \times \cdots \times [\hat{0},F_\ell].
    \]
    On the right side, the Cartesian product of posets
    inherits a componentwise ordering from its factors.
\end{defn}

\begin{example}
The \emph{maximal building set} $\G_{\max}:=\L \setminus \{\hat{0}\}$ vacuously satisfies the conditions for a building set, as one always has $\max (\G_{\max})_{\leq F}=\{F\}$. 
At the other extreme, the \emph{minimal building set} $\G_{\min}$ is the set of {\it indecomposable} elements: those whose lower interval $[\hat{0}, F]$ cannot be decomposed as a nontrivial product.

 Note that the {\it atoms} of the lattice, which are
 the elements covering $\hat{0}$, are always indecomposable, and hence included in every building set.
\end{example}

\subsection{Nested sets}
\label{nested-set-subsection}
Each building set $\G \subseteq \L \setminus \{\hat{0}\}$ gives rise to a family of subsets of $\G$
called $\G$-\emph{nested sets}.

\begin{defn}\label{def:nested-sets}
    A subset $N \subset \G$
    is called $\G$-\emph{nested} if 
    for any pairwise incomparable $F_1, F_2, \ldots, F_t \in N$, their join $F_1\vee F_2\vee \cdots \vee F_t$ is not in $\G$. 
    Denote by $\N(\L,\G)$ the collection
    of all $\G$-nested sets.
\end{defn}

Note that $\N(\L,\G)$ is closed under
inclusion, forming an abstract simplicial
complex on vertex set $\G$. In general, every totally ordered subset $N \subset \G$ is vacuously $\G$-nested. 

\begin{example}
When $\G = \G_{\max}$, the $\G$-nested sets are the {\it chains} (=linearly ordered subsets) of $\L \setminus \{\hat{0}\}$. The collection of nested sets, $\N(\L,\G_{\max})$, forms the {\it order complex} for the poset $\L \setminus \{\hat{0}\}$.
\end{example}

\begin{example}
\label{partition-lattice-example}
    Let $\L= \Pi_n$, the {\it lattice of set partitions} of $\{1,2,\ldots,n\}$. 
    The indecomposable elements of this lattice are those set partitions having at most one non-singleton block. After removing the decomposable elements, the remaining lattice now has the Petersen graph as its underlying graph.
    The lattice $\L=\Pi_4$ is displayed on the left below.  The indecomposable elements are in black, and the decomposables in gray.

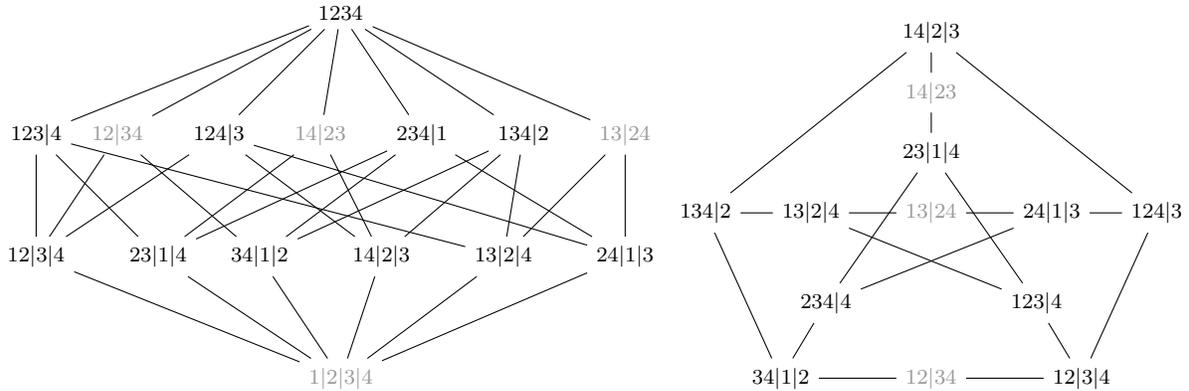
\begin{figure}[h]  
\centering 
\begin{tikzpicture} [scale=.27,auto=left,every node]
   \node[gray!80] (0) at (10,-1) {\footnotesize $1\vert 2\vert 3\vert 4$};
  
  \node (1) at (-5,5) {\footnotesize $12\vert 3\vert 4$};
  \node (2) at (1,5) {\footnotesize $ 23\vert1\vert 4$};
  \node (3) at (6,5) {\footnotesize $34\vert 1\vert 2$};
  \node (4) at (12,5) {\footnotesize $14\vert 2\vert 3$};
  \node (5) at (18,5) {\footnotesize $13\vert 2\vert 4$};
  \node (6) at (24,5) {\footnotesize $24 \vert 1\vert 3$};
  
  \node (125) at (-5,11) {\footnotesize $123\vert 4$};
  \node[gray!80] (13) at (-1,11) {\footnotesize $12\vert 34$};
  \node (146) at (4,11) {\footnotesize $124\vert 3$};
  \node[gray!80] (24) at (9,11) {\footnotesize $14\vert 23$};
  \node (236) at (14,11) {\footnotesize $234\vert 1$};
  \node (345) at (19, 11) {\footnotesize $134\vert 2$};
  \node[gray!80] (56) at (24, 11) {\footnotesize $13\vert 24$};
  \node (e) at (10,17) {\footnotesize $1234$};
  \foreach \from/\to in {0/1,0/2,0/3, 0/4, 0/5, 0/6,
  1/125, 1/13,1/146, 
  2/125, 2/24, 2/236,
  3/13, 3/236, 3/345,
  4/146, 4/24, 4/345,
   5/125, 5/345, 5/56,
   6/146, 6/236, 6/56,
  e/125, e/13, e/146, e/24, e/56, e/345, e/236}
    \draw (\from) -- (\to);
\end{tikzpicture}
\begin{tikzpicture}
 [scale=.20,auto=left,every node]
   \node[gray!80] (13-24) at (0,1) {\footnotesize $13|24$};
  \node (134) at (-15,1) {\footnotesize $134|2$};
  \node (13) at (-8,1) {\footnotesize $13|2|4$};
  \node (24) at (8,1) {\footnotesize $24|1|3$};
  \node (124) at (15,1) {\footnotesize $124|3$};
  \node (14) at (0,13) {\footnotesize $14|2|3$};
  \node[gray!80] (14-23) at (0,9) {\footnotesize $14|23$};
  \node (23) at (0,5) {\footnotesize $23|1|4$};
  \node (234) at (-7,-5) {\footnotesize $234|1$};
  \node (123) at (7,-5) {\footnotesize $123|4$};
  \node (34) at (-10,-10) {\footnotesize $34|1|2$};
  \node[gray!80] (12-34) at (0,-10) {\footnotesize $12|34$};
  \node (12) at (10,-10) {\footnotesize $12|3|4$};
  \foreach \from/\to in {134/13,13/13-24,13-24/24,24/124,134/14,14/124,124/12,12/12-34,12-34/34,34/134,34/234,234/24,14/14-23,14-23/23,23/234,23/123,13/123,12/123}
    \draw (\from) -- (\to);
\end{tikzpicture}
\end{figure}

The nested set complexes $\N(\L,\G_{\max}), \N(\L,\G_{\min})$ are both two-dimensional simplicial complexes which are cones with
apex/cone vertex $1234$; the base of the cone is the one-dimensional complex on the right above in the case of $\N(\L,\G_{\max})$.  For  $\N(\L,\G_{\min})$, the base of the cone is obtained from the same picture by {\it unsubdividing}  three edges, removing the vertices that correspond to the
decomposables $12|34, 13|24, 14|23$:
\begin{align*}
34|1|2 - 12|34 - 12|3|4 &\quad  \rightsquigarrow \quad 34|1|2 - 12|3|4,\\
13|2|4 - 13|24 - 24|1|3 &\quad \rightsquigarrow \quad 13|2|4 - 24|1|3,\\
14|2|3 - 14|23 - 23|1|4 &\quad \rightsquigarrow \quad 14|2|3 - 23|1|4.\\
\end{align*}
In this $n=4$ case, the graph on the right above is homeomorphic to the {\it Petersen graph}; compare these figures with those in 
\cite[Fig.~2]{ArdilaKlivans}, 
\cite[Fig.~2]{FeichtnerSturmfels}, 
\cite[p.~21]{FY}.  See also Example~\ref{ex:maximal-and-minimal-versus-hypotheses} for further discussion.
\end{example}

Interestingly, the restriction of the poset structure on $\L$ to any of its $\G$-nested sets $N$ becomes very simple.  Recall that a {\it forest poset} (as in, e.g., Bj\"orner and Wachs \cite{BjornerWachs}) is a finite poset in which each element is covered by at most one other element of the poset.  Note that this is {\it stronger} than simply requiring that the Hasse diagram of the poset is a forest, that is, an acyclic graph. 

\begin{lemma}\label{lem:rooted-forest}
 \cite[Prop.~2.4(4)]{PagariaPezzoli}
 For any building set $\G \subseteq \L \setminus \{\hat{0}\}$,
 the poset $\L|_N$ for any $\G$-nested set $N$ is always a forest poset.
 In particular, incomparable $F,F'$ in $N$ have disjoint principal ideals  $N_{\leq F}, N_{\leq F'}$, that is, $N_{\leq F} \cap N_{\leq F'}=\varnothing$.
\end{lemma}

\subsection{Atomic and geometric lattices and matroids}
\label{matroid-background-section}
So far these notions have played no role.  However,
atomicity is important in Feichtner and Yuzvinsky's  Gr\"obner basis presentation for their Chow ring $A(\L,\G)$, while the geometric lattice/matroid assumption implies a crucial numerical identity, Lemma~\ref{lem:crucial-numerical-fact} below.

\begin{defn}
Recall the {\it atoms} of a finite lattice $\L$ are
the elements that cover $\hat{0}$.  One calls $\L$ {\it atomic} if every $F$ in $L$ is the join of the set of atoms lying weakly below $F$ in $\L$.
\end{defn}

When $\L$ is a finite {\it atomic} lattice, the following quantity
is well-defined
for any $F,F'$ in $\L$ with $F \leq F'$:
\begin{equation}
\label{atomic-lattice-metric}
d(F,F')= \min \{d: F'=F \vee a_1 \vee \cdots \vee a_d \text{ for some atoms }a_1,\ldots,a_d \in \L\}.
\end{equation}
Well-definedness follows from the fact that
it is a minimum over a (finite) set of positive integers which is nonempty, as $F'=F \vee a_1 \vee \cdots \vee a_d$ where $\{a_1,\ldots,a_d\}$
is the set of {\it all} atoms of $L$ 
with $a_i \leq F'$.

This function $d(F,F')$ behaves more simply under extra
hypotheses on $\L$.

\begin{defn}
    A finite lattice $\L$ is {\it semimodular} if  $F,F'$ covering $F \wedge F'$ implies
    $F \vee F'$ covers $F, F'$.
\end{defn}

It is well-known \cite[Prop.~3.3.2]{Stanley-EC1} that, for a finite lattice $\L$, this condition is equivalent to being {\it ranked} and having the rank function $\rk: L \rightarrow \{0,1,2,\ldots\}$ obey the {\it semimodular inequality}
\begin{equation}
\label{semimodular-inequality}
\rk(F \vee F') + \rk(F \wedge F') \leq \rk(F)+\rk(F').
\end{equation}

\begin{defn}
    An atomic semimodular lattice $\L$ is called a {\it geometric lattice}.
\end{defn}
\noindent For geometric lattices $\L$, one can check that
the semimodular inequality \eqref{semimodular-inequality} implies this rewriting of \eqref{atomic-lattice-metric}:
\begin{equation}
\label{metric-rewritten}
d(F,F')=\rk(F')-\rk(F).
\end{equation}
\noindent
Finite geometric lattices are known \cite[\S3.3]{Stanley-EC1} to be the same as lattices of flats of finite {\it matroids}, briefly reviewed here. 
On matroids, some excellent references include Oxley \cite{oxley} and Ardila \cite{icm_matroids}.

\begin{defn} \label{flats} \rm
A matroid $\M = (\E, \L_\M)$ consists of a (finite) ground set $\E$ and a collection of subsets $\L_\M =\{F\} \subseteq 2^\E$ called {\it flats},
satisfying these axioms:
\begin{enumerate}
    \item[(F1)] $\E\in \L_\M$. 
    \item[(F2)] If $F, F'\in \L_\M$ , then $F \cap F' \in \L_\M$. 
    \item[(F3)] For any $F\in \L_\M$, and any $i \in \E\setminus F$, there is a unique $F' \in \L_\M$ containing $i$ which {\it covers $F$} in this sense: 
    $F \subsetneq F'$, and no other flat $F''$ has $F \subsetneq F'' \subsetneq F'$. 
\end{enumerate} 
\end{defn}
These axioms combinatorially abstract properties of a {\it realizable matroid} $\M$ coming from a list of vectors $v_1,v_2,\ldots,v_n$
in a vector space over some field $\mathbb{F}$.  
A realizable matroid $\M$ has ground set $E=\{1,2,\ldots,n\}$, and $F \subseteq E$ is a flat in $\L_\M$ when 
the inclusion $F \subseteq 
 \{i \in E: v_i \in \spn_{\mathbb{F}}\{ v_j\}_{j \in F} 
 \}$ is an equality. 
Axioms (F1),(F2) imply that the inclusion poset 
on $\L_\M$ forms a lattice, with {\it meets} $F \wedge F'=F \cap F'$ and {\it joins}
$F \vee F'=\bigcap_{F'' \supseteq F,F'} F''$.
Axiom (F3) further implies that the lattice $\L_\M$ will be {\it geometric}.
The {\it rank} of the matroid $\M$ itself is defined to be $\rk(E)$, which we assume throughout has $\rk(E)=r+1$. 

\begin{remark}
In this paper, we will only consider {\it simple matroids}, that is, 
\begin{itemize}
    \item $\M$ is {\it loopless}, so that $\hat{0}=\bigcap_{F \in \L_\M} F =\varnothing$ is a flat, and 
    \item $\M$ has no {\it parallel elements}, so that the atoms of $\L_\M$ are all singleton flats $F=\{a\}$.
\end{itemize}
This means that $\M$ is completely determined by
the geometric lattice $\L=\L_\M$.  This does not
impose a serious restriction, since the Chow ring $A(\L,\G)$ defined in Section~\ref{Chow-ring-section} below only uses information about the lattice strucutre of 
$\L$ and the choice of
building set $\G \subseteq \L \setminus \{\hat{0}\}$.  
\end{remark}

 The matroid automorphisms $\Aut(\M)$ of $\M$ are the permutations $g$ of the ground set $E$ preserving $\L_\M$ setwise, that is, sending flats $F$ to flats $g(F)$.  These also give poset and
lattice automorphisms of $\L_\M$. For a subgroup $\subgroup \subseteq \Aut(\M)$, a building set $\G$ is called \textit{$\subgroup$-stable} if each $g \in \subgroup$ preserves $\G$ setwise.

\begin{example}
\label{ex:maximal-and-minimal-versus-hypotheses}
For any $\subgroup \subseteq \Aut(\M)$, the maximal and minimal building sets 
$\G_{\max}, \G_{\min}$ are both $\subgroup$-stable.  
On the other hand, Theorem~\ref{main-theorem} imposes two extra technical conditions, the first being that the building set contains the ground set $E=\hat{1}$.
This always holds for $\G_{\max}$.  It holds for $\G_{\min}$ if and only if $\L=\L_\M$ for a {\it connected} (simple) matroid $\M$, that is, one cannot express $\M = \M_1 \oplus \M_2$ as a nontrivial direct sum.  

The second technical condition is the
stabilizer condition \eqref{eq:stabilizer-condition}, saying that a nested set which is setwise stabilized by some 
$g$ in $G$ must also be pointwise stabilized.  Again this always holds for $\G_{\max}$.  However for $\G_{\min}$, this condition can easily fail.  One such example of this type is the Chow ring for the {\it moduli space $\overline{M}_{0,n+1}$ of genus $0$ curves with $n+1$ marked points}; see, e.g., Dotsenko \cite{dotsenko}, Gibney and Maclagan \cite{GibneyMaclagan},  Keel \cite{Keel}, Maestroni and McCullough \cite[Rmk.~2.4]{MaestroniMcCullough}. This Chow ring turns out to be
$A(\L_\M,\G_{\min})$ where $\M$ is the graphic matroid associated to a complete graph on $n$ vertices and $\L_\M=\Pi_n$ is the set partition lattice from Example~\ref{partition-lattice-example}. In Remark~\ref{rmk:conclusion-failure} below, we explain how both the stabilizer condition \eqref{eq:stabilizer-condition} and the conclusion of Theorem~\ref{main-theorem} fail for this Chow ring when $n=12$.


\end{example}

\subsection{A surprisingly important identity} \label{suprising-identity-section}
A certain quantity from
Feichtner and Yuzvinsky \cite[\S3]{FY},
controls their 
Gr\"obner bases: for 
an atomic lattice $\L$, building set $\G$, and $\G$-nested set $N$ with $F \in N$, define
\begin{equation}
\label{atomic-version-of-crucial-quantity}
m_N(F):=d(\vee N_{<F},F)
\end{equation}
where $N_{<F}:=\{F' \in N: F' < F\}$.
For geometric lattices, \eqref{metric-rewritten} lets one rewrite \eqref{atomic-version-of-crucial-quantity} as in the Introduction:
\begin{equation}
\label{crucial-quantity}
m_N(F):=\rk(F) - \rk \left( \vee N_{<F} \right).
\end{equation}

The next identity involving $m_N(F)$ for geometric lattices plays a surprisingly crucial role here.
It appears in the proof of Corollary~\ref{Chow-ring-degrees} bounding the degrees occurring in a Chow ring.
It similarly appears in the proof of Theorem~\ref{main-theorem},
showing certain subsets of FY-monomials, forming products of chains under divisibility, live in intervals of degrees {\it symmetrically placed} with respect to the grading on
the Chow ring $A(\M,\G)$.

\begin{lemma}\label{lem:crucial-numerical-fact}
For any geometric lattice $\L$, any building set $\G$,
and any $\G$-nested set $N$, its join 
$\vee N$ has its rank expressible as
$
  \rk(\vee N)= \sum_{F \in N} m_N(F).
$
\end{lemma}

\begin{proof}
Proceed by induction on $|N|$.  The base case $N=\varnothing$ has $\rk(\vee N)=\rk(\hat{0})=0$ equal to the empty sum.
In the inductive step, let the set of inclusion-maximal elements
$\max(N)=\{F_1,\ldots,F_\ell\}$ 
have size $\ell$, and proceed by cases $\ell=1$
versus $\ell \geq 2$.

\vskip.1in
\noindent
{\sf Case 1.} $\ell=1$.

Here $F_1=\vee N$.
Letting $\hat{N}$ denote the $\G$-nested set $N \setminus \{F_1\}$, one has this calculation:
$$
\sum_{F \in N} m_N(F)
= m_N(F_1) + \sum_{F \in \hat{N}} m_N(F)
= m_N(F_1) + \sum_{F \in \hat{N}} m_{\hat{N}}(F) 
\overset{(*)}{=} m_N(F_1)+ \rk(\vee \hat{N}) 
=\rk(F_1)=\rk(\vee N)
$$
where equality (*) applied the inductive hypothesis to the nested set $\hat{N}$.

\vskip.1in
\noindent
{\sf Case 2.} $\ell \geq 2$.

From \cite[Prop 2.8]{FK}, one has that 
$
\max\G_{\leq \vee N} = \{F_1, F_2, \ldots, F_\ell\},
$
and therefore
$[\hat{0},\vee N] \cong \prod_{i=1}^\ell [\hat{0},F_i]$.
This gives the first equality in the following
calculation, whose remaining equalities are justified below:
\begin{align*}
\rk(\vee N) 
    = \sum_{i=1}^\ell \rk(F_i) 
    \overset{(**)}{=} \sum_{i=1}^\ell  \sum_{F \in N_{\leq F_i}} m_{N_{\leq F_i}}(F)  
    = \sum_{i=1}^\ell \sum_{F \in N_{\leq F_i}} m_N(F) 
    \overset{(***)}{=} \sum_{F \in N} m_N(F) 
\end{align*}
Equality (**) used the assumption $\ell \geq 2$, to apply the inductive hypothesis to each $\G$-nested set $N_{\leq F_i}$. Equality (***) used the fact that one has a disjoint decomposition
$N=\sqcup_{i=1}^\ell N_{\leq F_i}$
by Lemma~\ref{lem:rooted-forest}.
\end{proof}

\subsection{Chow Rings}
\label{Chow-ring-section}
The following definition was introduced by
Feichtner and Yuzvinsky \cite{FY}.

\begin{defn} 
\label{Chow-ring-definition}
For a finite lattice $\L$ and building set $\G$, 
the \emph{Chow ring} $A(\L, \G)$ is the quotient $\Z$-algebra 
    $$
    A(\L, \G) := S/(I+J)
    $$
    where $S=\Z[x_F]_{F \in \G}$ is a polynomial algebra, and $I,J$ are the following ideals of $S$:
    \begin{itemize}
      \item $I$ is the {\it Stanley-Reisner ideal} for the simplicial complex $\N(\L,\G)$, meaning it is
      generated by products $x_{F_1} x_{F_2} \cdots x_{F_t}$ for subsets  $\{F_1, F_2, \ldots, F_t\} \subset \G$ which are {\it not} $\G$-nested.
    \item $J$ is generated by the linear elements
$\displaystyle\sum_{a \in F \in \G} x_F$ for each atom $a$ 
in $\L$.
    \end{itemize}
\end{defn}

We wish to consider subgroups $\subgroup$ of the group $\Aut(\L)$ of poset automorphisms of $\L$, and building sets $\G \subseteq \L \setminus \{\hat{0}\}$ setwise stabilized by $\subgroup$, meaning $g(F) \in \G$ for all $F \in \G$ and $g$ in $\subgroup$.  We claim that in this setting, any $g$ in $\subgroup$ also acts by a $\Z$-algebra automorphism on $A(\L,\G)$ via $g(x_F)=x_{g(F)}$;  the key points are that poset automorphisms $g$ of $\L$
\begin{itemize}
\item preserve order and carry atoms to atoms, thus sending generators of $J$ to other generators of $J$, and 
\item because they
are also lattice automorphisms, they carry $\G$-nested sets to $\G$-nested sets, thus sending generators of $I$ to other generators of $I$.
\end{itemize}

We will also wish to consider $A(\L,\G)$ as a {\it graded} $\Z$-algebra, in which it inherits the standard grading on
$S=\Z[x_F]_{F \in \G}$ for which each variable $x_F$ has
$\deg(x_F)=1$.  This happens because the ideals $I,J \subset S$ are both graded ideals, generated by homogeneous elements of $S$. Also, subgroups $\subgroup \subset \Aut(\L)$ act as above via graded $\Z$-algebra automorphisms, that is, 
$A(\L,\G)=\bigoplus_{k=0}^\infty A^k$
and each $A^k$ becomes a $\Z \subgroup$-module.

It is not yet clear that $A(\L, \G)$ has only
finitely many non-vanishing components $A^k$.
For this, we next 
consider Feichtner and Yuzvinsky's remarkable Gr\"obner basis presentation
mentioned in the Introduction.

\subsection{In praise of the Feichtner--Yuzvinsky Gr\"obner basis}

We recall here a version of Gr\"obner basis theory with respect to monomial orders on polynomial rings 
$
\Z[\xx]:=\Z[x_1,\ldots,x_n]
$
over $\Z$, used in \cite{FY}.

\begin{defn}\rm
A linear order $\prec$ on the set of all monomials $\{ \xx^\alpha:=x_1^{\alpha_1} \cdots x_n^{\alpha_n}\}$ in $\Z[\xx]$ is called a {\it monomial ordering}
if it is a {\it well-ordering} (every subset of monomials has a $\prec$-minimum element) and $\xx^\alpha \prec \xx^\beta$ implies
$\xx^\alpha \cdot \xx^\gamma \prec \xx^\beta \cdot \xx^\gamma$
for all $\alpha,\beta,\gamma$ in $\{0,1,2,\ldots\}^n$.
\end{defn}

\begin{example} \rm
After fixing a linear order on the variables $x_1 < \cdots < x_n$,
the associated {\it lexicographic order} $\prec$ has
$\xx^\alpha \prec \xx^\beta$ if there exists some $k=1,2,\ldots,n$ with  $\alpha_1 = \beta_1, \alpha_2=\beta_2,\ldots,\alpha_{k-1} = \beta_{k-1}$,
but $\alpha_k < \beta_k$.
\end{example}

\begin{defn}\rm
For $f=\sum_\alpha c_\alpha \xx^\alpha$ in $\Z[x_1,\ldots,x_n]$,
let $\init_\prec(f)=\xx^\beta$ denote the $\prec$-largest monomial in $f$ having $c_\beta \neq 0$.  Say $f$ is {\it $\prec$-monic} if the initial monomial $\init_\prec(f)=\xx^\beta$ has its coefficient $c_\beta=\pm 1$.

Having fixed a monomial order $\prec$, for any ideal $I \subset \Z[\xx]$, one says that a subset  $\{g_1,\ldots,g_t\} \subset I$ is a {\it monic Gr\"obner basis} for $I$ (with respect to $\prec$) if each $g_i$ is $\prec$-monic, and 
every $f$ in $I$ has $\init_\prec(f)$ divisible by at least one
of the intial monomials 
$\{ \init_\prec(g_1),\ldots,\init_\prec(g_t)\}$.
One calls $\xx^\alpha$ a {\it standard monomial} for 
$\{g_1,\ldots,g_t\}$ with respect to $\prec$ if
$\xx^\alpha$  is divisible by none of $\{ \init_\prec(g_1),\ldots,\init_\prec(g_t)\}$.
\end{defn}

It should be noted that some ideals $I \subset \Z[\xx]$
have {\it no} monic Gr\"obner basis, e.g. $I=(2) \subset \Z[x]$.  However, whenever $I$ {\it does} have a monic Gr\"obner basis, it has the following strong consequence, proven exactly as for Gr\"obner bases over
$\mathbb{F}[\xx]$ with field cofficients $\mathbb{F}$; see, e.g.,
Cox, Little, O'Shea \cite[\S2.5, 5.3]{CLO}.
\begin{prop}
\label{GB-basis-fact}
Given a monic Gr\"obner basis $\{g_1,\ldots,g_t\}$ for $I \subset S=\Z[\xx]$ with respect to $\prec$, then
\begin{itemize}
\item $I=(g_1,\ldots,g_t)$, that is, $\{g_1,\ldots,g_t\}$ generate $I$ as an ideal.
\item The quotient ring $S/I$ is a free $\Z$-module, with
a $\Z$-basis given by the standard monomials $\{x^\alpha\}$.
\end{itemize}
\end{prop}

The following crucial result appears as \cite[Thm. 2]{FY}.  To state it, define an {\it FY-monomial order} on $S=\Z[x_F]_{F \in \G}$ to be any monomial order
based on a linear order of the variables with $x_F > x_{F'}$ if $F < F'$.

\begin{thm}
\label{FY-GB-theorem}
Let $\L$ be a finite atomic lattice, and $\G \subset \L \setminus \{\hat{0}\}$ a building set.
Then for any choice of an $\FY$-monomial order on the ring $S=\Z[x_F]_{F\in \G}$,
the ideal $I+J$ presenting  $A(\L, \G)=S/(I+J)$
has a monic Gr\"obner basis consisting of the following
elements, along with their $\prec$-initial terms:
\begin{itemize}
\item
all $x_{F_1} x_{F_2} \cdots x_{F_t}$ for sets $\{F_1,F_2,\ldots,F_t\} \subset \G$ which are {\it not} $\G$-nested, each its own $\prec$-initial term,
\item 
for each $\G$-nested set $N$ with maximal elements
$\{F_1,\ldots,F_\ell\}$, and for each $F$ in $\G$ with $F > \vee N$,
\begin{align*}
\text{ the product }\quad x_{F_1} \cdots x_{F_\ell} 
&\left( \sum_{F' \geq F} x_{F'} \right)^
{d(\vee N, F)}\\
\text{ whose $\prec$-initial term is }\quad  
x_{F_1} \cdots x_{F_\ell} \cdot  &x_F^
{d(\vee N, F)}.
\end{align*}
\end{itemize}
\end{thm}

This has the following very strong consequence,
using the notation $m_N(F):=d(\vee N_{<F}, F)$ defined in \eqref{atomic-version-of-crucial-quantity}, where
$d(F,F')=\min\{d:F'=F \vee a_1 \vee \cdots \vee a_d \text{ for atoms }a_i \in \L\}$ was
defined in \eqref{atomic-lattice-metric}.

\begin{cor} (\cite[Cor. 1]{FY}) 
\label{cor: mon_basis}
Under the hypotheses of Theorem~\ref{FY-GB-theorem},
the Chow ring $A(\L,\G)$ is free as a $\Z$-module, with
the $\prec$-standard monomial $\Z$-basis given by the set of $\FY$-monomials
\begin{equation}
\label{FY-monomials-definition}
\fy := \left\{ x_{F_1}^{m_1} \cdots x_{F_\ell}^{m_\ell}: 
N=\{F_1,\ldots,F_\ell\} \text{ is }\G\text{-nested, and }0 \leq m_i <m_N(F_i)\right\}.
\end{equation}
\end{cor}

This has a particularly interesting further consequence for symmetries of $\L$,
inspired by work of H.-C. Liao on the Chow rings of Boolean lattices $\L$
with their maximal building set $\G_{\max}$; see \cite[Thm. 3.6]{Liao_new}.

\begin{cor}
\label{Chow-ring-carries-perm-reps-cor}
Under the hypotheses of Theorem~\ref{FY-GB-theorem},
if one has a subgroup $\subgroup \subseteq
\Aut(\L)$ setwise stabilizing the
building set $\G$, then
$\subgroup$ permutes the set $\FY$,
as well as its subset of degree $k$ monomials $\FY^k \subset \FY$.  

Consequently, the $\Z \subgroup$-modules on the Chow ring $A(\L, \G)$ and each of its homogeneous components $A^k$ lift to $\subgroup$-permutation representations on $\FY$ and each $\FY^k$.  
\end{cor}
\begin{proof}
It was noted that $g$ in $\subgroup$ sends $\G$-nested sets $N$ to
$\G$-nested sets $g(N)$.  Since such $g$ are lattice automorphisms, they have $g(\vee N)=\vee g(N)$.  Also $d(g(F),g(F'))=d(F,F')$, so that
$m_{g(N)}(g(F))=m_N(F)$.
Thus $g$ sends $\fy$-monomials
$x_{F_1}^{m_1} \cdots x_{F_\ell}^{m_\ell}$
to other $\fy$-monomials
$x_{g(F_1)}^{m_1} \cdots x_{g(F_\ell)}^{m_\ell}$.
\end{proof}

\begin{remark} \rm
It is rare to find families of ideals $I$ inside polynomial rings $S$ that are stable under a finite group $\subgroup$ acting on $S$, and which
{\it also have a $\subgroup$-stable Gr\"obner basis} $\{g_i\}$ 
with respect to some monomial order $\prec$.  This occurs, for example, with the {\it Hibi rings} studied in \cite{Hibi}, and the more general $P$-partition rings studied by Feray and the third author in \cite[Thm. 6.3]{FerayR}.  More often, when starting with a $\subgroup$-stable ideal $I$, passing to an intial ideal destroys the $\subgroup$-symmetry. One then
usually needs alternative techniques to work with the quotient $S/I$, as discussed, e.g., by Faug\'ere \cite{Faugere} and Sturmfels \cite[\S 2.6]{Sturmfels}.
\end{remark}

\begin{remark} \rm
Although the $\Z \subgroup$-module structure on $A(\L, \G)$ and $A^k$ are canonical, their lifts to permutation representations on the
sets $\FY$ and $\FY^k$ are not.  In general, one can have two different $\subgroup$-permutation representations on sets $X, X'$ (that is, with no $\subgroup$-equivariant set bijection $X \overset{\sim}{\longrightarrow} X$) but 
with a $\Z \subgroup$-module isomorphism $\Z [X] \cong \Z [X']$; 
see, e.g., Conlon \cite{Conlon}
and Scott \cite{Scott}. 

\end{remark}

When $\L$ is a geometric lattice, the foregoing results
give more precise degree bounds on $A(\L,\G)$. 

\begin{cor}
\label{Chow-ring-degrees}
Let $\L$ be a geometric lattice of rank $r+1$, and $\G \subseteq \L \setminus \{\hat{0}\}$ a building set.
\begin{itemize}
\item[(i)] $A(\L, \G)$ vanishes in degrees strictly above $r$,
that is, $A(\L, \G)=\bigoplus_{k=0}^r A^k$.
\item[(ii)] In fact, $A^r=0$ unless $\G$ contains $\hat{1}$, in which case $A^r$ has $\Z$-basis $\{x_{\hat{1}}^r\}$, and hence one has a $\Z$-module isomorphism
$\deg: A^r \longrightarrow \Z$ sending $x_{\hat{1}}^r \longmapsto 1$.
\end{itemize}
\end{cor}
\begin{proof}
For assertion (i), note that the typical $\FY$-monomial $x_{F_1}^{m_1} x_{F_2}^{m_2} \cdots x_{F_\ell}^{m_\ell}$,
has total degree
\begin{equation}
\label{bounding-FY-monomial-degree}
\sum_{i=1}^\ell m_i
\leq \sum_{i=1}^\ell (m_N(F_i) - 1)
\overset{(*)}{=} \rk(\vee N) - \ell
\leq (r+1) - 1 = r.
\end{equation}
where the inequality (*) used Lemma~\ref{lem:crucial-numerical-fact}.
For assertion (ii), note that
equality in \eqref{bounding-FY-monomial-degree} can occur only if $\ell=1$ and $\rk(\vee N)=r+1$. This 
forces $N=\{\hat{1}\}$, implying that the $\FY$-monomial must be $x_{\hat{1}}^r$.
\end{proof}

\begin{example}
Consider the Boolean lattice $\L=2^{\{1,2,\ldots,n\}}$ of rank $n$, along with these two building sets \begin{align*}
    \G_{\min}&=\{ \{1\},\{2\}, \ldots, \{n\} \},\\
    \G=\G_{\min} \cup \hat{1} &=\{ \{1\},\{2\}, \ldots, \{n\}, \{1,2,\ldots,n\} \}.
\end{align*}
One can check that their Chow rings have these descriptions:
\begin{align*}
    A(\L,\G_{\min})&=\Z[x_1,x_2,\ldots,x_n]/(x_1,x_2,\ldots,x_n)=\Z,\\
    A(\L,\G)&=\Z[x_1,x_2,\ldots,x_n,x_{\hat{1}}]
      /(x_1 x_2 \cdots x_n, \,\, x_1+x_{\hat{1}},\,\,
      x_2+x_{\hat{1}},\,\,\ldots,\,\, x_n+x_{\hat{1}})\\
      &\cong \Z[x_{\hat{1}}]/(x_{\hat{1}}^n)
      =\spn_\Z\{1,x_{\hat{1}},x_{\hat{1}}^2,
      \ldots,x_{\hat{1}}^r\}
\end{align*}
where we recall that $\rk(\L)=n=r+1$, so $r=n-1$.
\end{example}

\begin{remark}
    Here we compare the above definition of the Chow ring with the one used by Adiprasito, Huh and Katz \cite{AHK}. They study  Chow rings $A(\L,\G_{\max})$ for geometric lattices $\L=\L_\M$ with the choice of the maximal building set, but they alter the presentation of $A(\L_\M, \G_{\max})$ to eliminate the variable
    $x_E (=x_{\hat{1}})$ from the polynomial ring $S=\Z[x_F]_{\varnothing \subsetneq F \subseteq E}$.  Specifically they write $A(\L_M, \G_{\max})=\hat{S}/(\hat{I} + \hat{J})$ where $\hat{S}=\Z[x_F]_{\varnothing \subsetneq F \subsetneq E}$ has
    a variable $x_F$ for each nonempty, \emph{proper} flat $F$.  Then $\hat{I}$ is again the ideal generated by the monomials $x_{F} x_{F'}$ where $F, F'$ are not-nested, but $\hat{J}$ is now generated by the elements
    \[
    \displaystyle\sum_{\substack{F \in \L{_\M}:\\a \in F \neq E}} x_F
    - \displaystyle\sum_{\substack{F \in \L_\M:\\ a' \in F \neq E}} x_F
    \]
    for each pair of distinct atoms $a \neq a'$ in $\L_\M$.  One can check this presentation of $A(\L_\M, \G_{\max})$ is equivalent to 
    Definition~\ref{Chow-ring-definition}: mutually inverse isomorphisms are induced by the map $\hat{S} \rightarrow S$ sending $x_F \longmapsto x_F$, and the map $S \rightarrow \hat{S}$ sending $x_F \longmapsto x_F$ for $F \neq E$ and $x_E \longmapsto -\displaystyle\sum_{a \in F \neq E} x_F$ for any atom $a$ in $\L$.
\end{remark}

\begin{remark}
In \cite{BHMPW}, the authors introduce
the {\it augmented Chow ring} $\tilde{A}(\M)$ of a matroid $\M$ on ground set $E$,
and prove that it also satisfies the K\"ahler package.
It was observed by Chris Eur \big(as explained in \cite[\S4, p.~1843]{MaestroniMcCullough}; see also \cite[\S5]{EHL}, \cite[Thm.~5.10]{Liao_new}\big) that $\tilde{A}(\M)$ is isomorphic to an instance of the usual Chow ring for another related matroid. In particular, they showed that $\tilde{A}(\M) = A(\M \times e,\tilde{\G})$, where $\M \times e:=(\M^\perp + e)^\perp$ on the larger ground set $E \sqcup \{e\}$ is the {\it free coextension} of $\M$, and $\tilde{\G}$ is the following (non-maximal) building set 
for $\M \times e$:
$$
\tilde{\G}= \{\{1\},...,\{n\}\} \sqcup \{ F\cup \{e\}: \text{ for all flats  }F \in \L(M)\}.
$$
 Since $\Aut( \M \times e)$ contains $\Aut(\M)$ as a subgroup,
 and since the above building set $\tilde{\G}$ is $\Aut(\M)$-stable, our results will apply
also to the augmented Chow rings of $\M$ with
its action by $\Aut(\M)$.
\end{remark}

\subsection{The K\"ahler package}
\label{AHK-section}
The following theorem compiles
some of the main results of the work of
Adiprasito, Huh and Katz \cite{AHK}, who originally proved the K\"ahler package for Chow rings $A(\L,\G_{\max})$ of geometric lattices $\L$ with their maximal building sets $\G_{\max}$.
The result for other
building sets
follows from a result of Pagaria
and Pezzoli \cite[Thm.~4.21]{PagariaPezzoli}, applying even more generally to Chow rings of {\it polynmatroids} and their building sets.
See also Remark~\ref{Denham-remark} below
for the sketch of an alternate proof\footnote{The authors thank Chris Eur for suggesting that such a proof should exist, and Graham Denham for supplying the details.}
using results of Ardila, Denham and Huh \cite{ADH}.

\begin{thm}
\label{AHK-theorem}
For a simple matroid $\M$ with lattice of flats $\L=\L_\M$ of rank $r+1$, and any choice of building set $\G \subset \L \setminus \{\hat{0}\}$ that contains the ground set $\hat{1}=E$, the Chow ring $A(\L,\G)=\bigoplus_{k=0}^r A^k$ satisfies the 
K\"ahler package:
\begin{itemize}
\item (Poincar\'e duality)\\ 
For every $k \leq \frac{r}{2}$, 
one has a perfect $\Z$-bilinear pairing
\begin{align*}
    A^k \times A^{r-k} &\longrightarrow \Z \\ 
    (a,b) &\longmapsto \deg(a \cdot b)
\end{align*}
that is, $b \longmapsto \varphi_b(-)$ defined by
$\varphi_b(a)=\deg(a \cdot b)$ is a
$\Z$-linear isomorphism $A^{r-k} \cong \Hom_\Z(A^k,\Z)$.
\item[]
\item (Hard Lefschetz)\\
Tensoring over $\Z$ with $\R$, the (real) Chow ring $A_\R(\L,\G)=\sum_{k=0}^r A^k_\R$ contains {\textbf{Lefschetz elements}} $\omega$ in $A^1_\R$, meaning
that $a \mapsto a \cdot \omega^{r-2k}$ is an $\R$-linear isomorphism $A^k_\R \rightarrow A^{r-k}_\R$ for $k \leq \frac{r}{2}$.

In particular, multiplication by $\omega$ is an injection $A^k_\R \rightarrow A_\R^{k+1}$ for $k < \frac{r}{2}$.
\item[]
\item (Hodge-Riemann-Minkowski inequalities)\\
The Lefschetz elements $\omega$ define quadratic forms 
$a \longmapsto (-1)^k \deg(a \cdot \omega^{r-2k} \cdot a)$
on $A^k_\R$ that become positive definite 
upon restriction to the kernel of the map
$A^k_\R \longrightarrow A^{r-k+1}_\R$ that sends 
$a \longmapsto  a \cdot \omega^{r-2k+1}$.
\end{itemize}
\end{thm}

Now consider the situation where $\subgroup$ is a subgroup of $\Aut(\L)$ that setwise stabilizes $\G$, and hence acts on $A(\M, \G)$ and each $A^k$.  We point a corollary of the above result that will be refined by Theorem~\ref{main-theorem}(i).

\begin{cor}
\label{integral-equivariant-PD-cor}
If the building set $\G$ contains $\hat{1}=E$, in the above context, then one has an isomorphism of $\Z \subgroup$-modules 
$
A^{r-k} \rightarrow A^k
$
for each $k \leq \frac{r}{2}$.
\end{cor}
\begin{proof}
Corollary~\ref{Chow-ring-degrees} shows that $A^r$ has only one $\Z$-basis element $x_E^r$, fixed by every $g$ in $\subgroup$, so the degree map $\deg: A^r \longrightarrow \Z$ is $\subgroup$-equivariant for
the trivial $\subgroup$-action on the target $\Z$.  Thus the 
Poincar\'e duality isomorphism $A^{r-k} \longrightarrow \Hom_\Z(A^k,\Z)$, sending $b \longmapsto \varphi_b(-)$ with
$\varphi_b(a)=\deg(a \cdot b)$, is also $\subgroup$-equivariant.  

It only remains to exhibit a $\subgroup$-equivariant isomorphism $\Hom_\Z(A^k,\Z) \rightarrow A^k$.  To this end, use Corollary~\ref{Chow-ring-carries-perm-reps-cor} to pick a $\Z$-basis
$\{e_i\}$ permuted by $\subgroup$, so that each element $g$ in $\subgroup$ acts by a permutation matrix $P(g)$ in this basis.  Letting $\{f_i\}$ be the dual $\Z$-basis for $\Hom_\Z(A^k,\Z)$, one finds that $g$ acts via the matrix $P(g^{-1})^T = (P(g)^{-1})^T=P(g)$, since $P(g)$ is a permutation matrix.  Hence the map $e_i \longmapsto f_i$ is such a
$\subgroup$-equivariant isomorphism $\Hom_\Z(A^k,\Z) \rightarrow A^k$.
\end{proof}

The next consequence will be refined by Theorem~\ref{main-theorem}(ii).

\begin{cor}
\label{AHK-equivariant-Hard-Lefschetz}
In the context of Theorem~\ref{AHK-theorem}, where $\L=\L_\M$ for a simple matroid $\M$, with $\G$ any $\subgroup$-stable building set containing $\hat{1}=E$, one has $\R \subgroup$-module maps
    $A^k_\R \rightarrow A_\R^{k+1}$ which are
    injective for $k < \frac{r}{2}$.
\end{cor}
\begin{proof}
When there exists a
{\it $\subgroup$-fixed} Lefschetz element
$\omega$ in $A^1_\R$ in Theorem~\ref{AHK-theorem}, 
multiplication by $\omega$ gives the $\R \subgroup$-module injections.  For $\G=\G_{\max}$ the maximal building set,
Adiprasito, Huh and Katz \cite{AHK} exhibited an explicit such $\subgroup$-fixed $\omega=\sum_{\varnothing \neq F \in \L_\M} c_F x_F$
on \cite[p. 384]{AHK}, namely
by choosing $c_F=|F| \cdot |E \setminus F|$. 

For more general $\subgroup$-stable building sets, one can use
the fact that \cite[Thm.~4.12]{PagariaPezzoli} exhibits a (nonempty) and $G$-stable {\it convex cone} $K$ of Lefschetz elments $\omega$ inside $A^1(\L_\M,\G)$.  Starting with
any such $\omega$, one can replace it with its $\subgroup$-average\footnote{Thanks to Graham Denham for pointing out this averageing trick.} $\frac{1}{|G|} \sum_{g \in \subgroup} g(\omega)$, which is also
in the cone $K$,
and which is now a $\subgroup$-fixed Lefschetz element.
\end{proof}

\begin{remark}
\label{Denham-remark}
The authors thank G. Denham for sketching the following alternate proof of Theorem~\ref{AHK-theorem} and Corollary~\ref{AHK-equivariant-Hard-Lefschetz}, using results from Ardila, Denham and Huh \cite[\S5]{ADH}.
As in the above proof of Corollary~\ref{AHK-equivariant-Hard-Lefschetz}, for each $G$-stable building set $\G$, the point is to exhibit a nonempty, $G$-stable convex cone $\Sigma$ of Lefschetz elements in degree one $A_\R^1$ of $A_\R(\M,\G)$.  

In \cite[Defn.~5.1]{ADH}, the authors define the cone $\mathcal{K}(\Sigma) \subset A^1_\R \subset A_\R(\Sigma)$ of {\it strictly convex} elements for a rational simplicial fan $\Sigma$ and its Chow ring $A(\Sigma)$.  In
particular, the choice of a $G$-stable building set $\G$ for $\M$ leads to
such a rational simplicial fan $\Sigma_{\M,\G}$, which is $G$-stable,
and $A(\Sigma_{\M,\G})=A(\M,\G)$.
This fan $\Sigma_{\M,\G}$ is a simplicial subdivision of what is called the {\it Bergman fan} of the matroid $\M$.  As part of their proof of \cite[Thm.~6.1]{FeichtnerSturmfels},
Feichtner and Sturmfels show that
this subdivision $\Sigma_{\M,\G}$ of the Bergman fan of $\M$ is a subfan of
the normal fan of a {\it convex polytope}, namely
the {\it Minkowski sum} $\Delta_\G:=\sum_{F \in \G} \Delta_F$ of standard simplices $\Delta_F$ for each $F$ in $\G$. 
In the language of \cite[Rmk.~5.4]{ADH},
this means that $\Sigma_{\M,\G}$ is {\it quasiprojective}, and hence the
cone $\mathcal{K}(\Sigma_{\M,\G})$
is nonempty.  Applying \cite[Thm.~1.6]{ADH}, one concludes
that all elements $\omega$ of $\mathcal{K}(\Sigma_{\M,\G})$ are  Lefschetz elements for $A(\M,\G)$, 
since the same holds for the maximal
building set $\G=\G_{\max}$ by the results of \cite{AHK}.
\end{remark}

\section{Proof of Theorem~\ref{main-theorem}}
\label{main-theorem-section}
Recall the statement, involving a simple matroid $\M$
with lattice of flats $\L=\L_\M$, a choice of a building
set $\G \subset \L \setminus \{\hat{0}\}$ that
contains $\hat{1}=E \in \G$, a subgroup $G \subseteq \Aut(\L)$,
for which $\G$ is setwise $G$-stable, and which satisfies the
{\it stabilizer condition} \eqref{eq:stabilizer-condition}:
\begin{quote}
 for any $\G$-nested set
   $N=\{F_i\}_{i=1,\ldots,\ell}$, if $g \in G$  has 
    $g(N)=N$, then $g(F_i)=F_i$ for $i=1,\ldots,\ell.$
\end{quote}

The statement also uses
the $\Z$-basis of FY-monomials for the Chow ring
$A(\L, \G)$ from Corollary~\ref{cor: mon_basis}:
$$
\fy:=\left\{x_{F_1}^{m_1}  \cdots x_{F_\ell}^{m_\ell} \colon
N:=\{F_1, \cdots,  F_\ell\} \text{ is } \G\text{-nested, and } 
0 \leq m_i < m_N(F_i) \text{ for }i=1,2,\ldots,\ell.\right\}
$$
where $m_N(F):=\rk(F)-\rk(\vee N_{<F})$ was defined in the Introduction (and repeated in \eqref{crucial-quantity}).

\vskip.1in
\noindent
{\bf Theorem~\ref{main-theorem}.}
{\it 
In the above context, there exist
\begin{itemize}
\item[(i)]
$\subgroup$-equivariant  bijections 
$
\pi: \FY^k  \overset{\sim}{\longrightarrow}  \FY^{r-k}
$ for $k \leq \frac{r}{2}$, and
\item[(ii)]
$\subgroup$-equivariant injections
$
\lambda: \FY^k  \hookrightarrow  \FY^{k+1}
$
for $k < \frac{r}{2}$.
\end{itemize}
}
\vskip.1in


\begin{proof}[Proof of Theorem~\ref{main-theorem}]

The proof proceeds in these four steps:

\begin{itemize}
\item[{\sf Step 1.}]
    We first segregate the set of all of all $\FY$-monomials 
into the fibers of an {\it extended support} map
$$
\supp_+: \FY \rightarrow \N(\L,\G)
$$
defined as follows: if $m_i \geq 1$ for $1 \leq i \leq \ell$,
then
$$
\supp_+(x_{F_1}^{m_1} x_{F_2}^{m_2} \cdots x_{F_\ell}^{m_\ell})
:=\{F_1,\ldots,F_\ell\} \cup \{E\}.
$$ 
One of these fibers of $\supp_+$ is depicted in Example~\ref{extended-support-map-example} below.

\item[{\sf Step 2.}]
We check that for each 
$\G$-nested set $N^+=\{F_1,\ldots,F_\ell,E\}$ in the image of $\supp_+$,
the fiber $\supp_+^{-1}(N^+)$
is a collection of monomials which, when ordered by divisibility, forms a {\it Cartesian product of (saturated) chains}.  We also check, via Lemma~\ref{lem:crucial-numerical-fact}, that the degrees of the monomials in the fiber $\supp_+^{-1}(N^+)$ occupy a {\it symmetrically placed} interval $\{\ell,\ell+1,\ldots,(r-\ell)-1,r-\ell\}$
within the range of degrees $\{0,1,\ldots,r-1,r\}$ for the whole Chow ring $A(\L,\G)$.

\item[{\sf Step 3.}]
We check that the stabilizer condition \eqref{eq:stabilizer-condition} allows one to choose the identifications from {\sf Step 2} of
$\supp_+^{-1}(N^+)$ with Cartesian products of
chains {\it $G$-equivariantly} in a certain well-defined sense.

\item[{\sf Step 4.}]
Lastly, we show how the desired bijections $\pi$ and injections $\lambda$ on $\FY$-monomials can be pulled back from the corresponding self-dualities
and order-raising maps on the Cartesian products of chains.  The latter maps are derived from fixed {\it symmetric chain decompositions}, and
are checked to pull back to $G$-equivariant maps on the $\FY$-monomials, using the $G$-equivariance from {\sf Step 3}.

\end{itemize}

We now embark on the proof.
\vskip.1in
\noindent
{\sf Step 1.}\\
As in the outline, segregate the set of all $\FY$-monomials 
into the fibers of the {\it extended support} map
$$
\supp_+: \FY \rightarrow \N(\L,\G)
$$
defined as follows: if $m_i \geq 1$ for $1 \leq i \leq \ell$,
then
$$
\supp_+(x_{F_1}^{m_1} x_{F_2}^{m_2} \cdots x_{F_\ell}^{m_\ell})
:=\{F_1,\ldots,F_\ell\} \cup \{E\}.
$$ 
Note by Definition~\ref{def:nested-sets} of $\G$-nested sets
that, since $N=\{F_1,\ldots,F_\ell\}$ is $\G$-nested,
so is $N \cup \{E\}$. 

\vskip.1in
\noindent
{\sf Step 2.}\\
By the definition of $\fy$ in Corollary~\ref{cor: mon_basis}, a typical $\G$-nested set $N^+=\{F_1,\ldots,F_\ell,E\}$ in the image of $\supp_+$ 
has the following description for its fiber: 
$
\supp_+^{-1}(N^+) =\{
x_{F_1}^{m_1} x_{F_2}^{m_2} \cdots x_{F_\ell}^{m_\ell} x_E^{m_{\ell}+1}
\}
$
where the exponents on the monomials satisfy these inequalities:
\begin{equation}
    \label{fiber-inequalities}
\begin{aligned} 
&1 \leq m_{i} \leq m_{N^+}(F_i)-1 \text{ for }1\leq i \leq \ell,\\
&0 \leq m_{\ell+1} \leq m_{N^+}(E) - 1.
\end{aligned}
\end{equation}
As a consequence, the degrees of monomials
in the fiber $\supp_+^{-1}(N^+)$ lie in the range $[\ell,r-\ell]$: 
\begin{itemize}
    \item the minimum degree is achieved by $\deg(x_{F_1}^1 \cdots x_{F_\ell}^1 x_{E}^0)=\ell$, and
    \item the
maximum degree is achieved by
\begin{align*}
\deg\left(x_{F_1}^{m_{N^+}(F_1)-1} 
           \cdots x_{F_\ell}^{m_{N^+}(F_\ell)-1}
             x_{E}^{m_{N^+}(E)-1} \right)
&= \sum_{F \in N^+} (m_{N^+}(F)-1)\\
&= \rk(\vee N^+) - |N^+| \quad \text{ using Lemma~\ref{lem:crucial-numerical-fact},}\\
&= (r+1)-(\ell+1)=r-\ell. 
\end{align*}
\end{itemize}
Note also that the set of monomials in each such fiber $\supp_+^{-1}(N^+)$, when ordered via divisibility,
gives a poset isomorphic to a Cartesian product of chains
\begin{equation}
    \label{typical-product-of-chains}
\supp_+^{-1}(N^+) \cong C_{m_{N^+}(F_1)-1} \times 
C_{m_{N^+}(F_2)-1} \times \cdots \times
C_{m_{N^+}(F_\ell)-1} \times
C_{m_{N^+}(E)},
\end{equation}
where $C_m$ denotes a chain (totally ordered set) of size $m$.

\vskip.1in
\noindent
{\sf Step 3.}\\
In order to pin down a choice of the poset isomorphism in \eqref{typical-product-of-chains}, for each $\G$-nested set $N^+=\{F_1,\ldots,F_\ell,E\}$ in the image of $\supp_+$,
one must choose a linear order $(F_1,\ldots,F_\ell)$. It will be important
for the proof that these choices of linear order are {\it $G$-equivariant} in the 
following sense: whenever one has two $\G$-nested sets $N$ and $N'$ with chosen
linear orders $(F_1,F_2,\ldots,F_\ell)$ and $(F_1',F_2',\ldots,F_\ell')$, if
$g \in G$ has $g(N)=N'$, then one also has $g(F_i)=F_i'$ for $i=1,2,\ldots,\ell$. 
The fact that such a $G$-equivariant choice is possible follows from the stabilizer condition \eqref{eq:stabilizer-condition}:  one can choose the linear order $(F_1,\ldots,F_\ell)$ arbitrarily
on one $G$-orbit representative $N=\{F_1,\ldots,F_\ell\}$ from each $G$-orbit of $\G$-nested sets,
and then for any other element $g(N)$ within the orbit, decree its linear order to
be $(g(F_1),\ldots,g(F_\ell))$.  Condition \eqref{eq:stabilizer-condition} ensures that this is well-defined, independent of $g$: if $g,h \in G$ have $g(N)=h(N)$, then $g^{-1}h(N)=N$, so $g^{-1}h(F_i)=F_i$ and $g(F_i)=h(F_i)$ for
$i=1,\ldots,\ell$.

\vskip.1in
\noindent
{\sf Step 4.}\\
We now use the fact that
every ranked poset $P$ of rank $R$ which is a product of chains has a {\it symmetric chain decomposition (SCD)},
meaning a disjoint decomposition $P=\bigsqcup_{i=1}^t P_i$
    in which each $P_i$ is a chain passing through ranks $\{\rho_i,\rho_i+1,\ldots,R-\rho_i-1,R-\rho_i\}$ for some $\rho_i \in \{0,1,2,\ldots,\lfloor \frac{R}{2}\rfloor \}$;
    see, e.g., Anderson \cite[\S 3.1]{Anderson}.
Fix one such SCD for each product poset on the right side in \eqref{typical-product-of-chains}, which together with the choice of linear orders $(F_1,\ldots,F_
\ell)$, will induce an SCD on each fiber poset $\supp_+^{-1}\{F_1,\ldots,F_\ell,E\}$.  
An important point for the remainder of the proof
is that the structure of these symmetric chains depend only upon
the the linear orders $(F_1,\ldots,F_\ell)$ and numerical sequences 
$(m_{N^+}(F_1),  m_{N^+}(F_2),
\ldots,
m_{N^+}(F_\ell),
m_{N^+}(E))$, rather than depending upon finer information about the $\G$-nested set $N^+$.

\begin{example} \rm
\label{extended-support-map-example}
Let $\L=\L_\M$ be the lattice flats of a simple matroid $\M$ having $\rk(E)=10=r+1$ with $r=9$, and $\G=\G_{\max}$.  Consider
a $\G$-nested set $N^+=\{F_1,F_2,E\}$ where $F_1 < F_2$ are flats with $\rk(F_1)=3, \rk(F_2)=7$.  Linearly ordering $\{F_1,F_2\}$ as $(F_1,F_2)$, then the
divisibility poset on the fiber $\supp_+^{-1}\{F_1,F_2,E\}$ and one choice of SCD look as follows: 


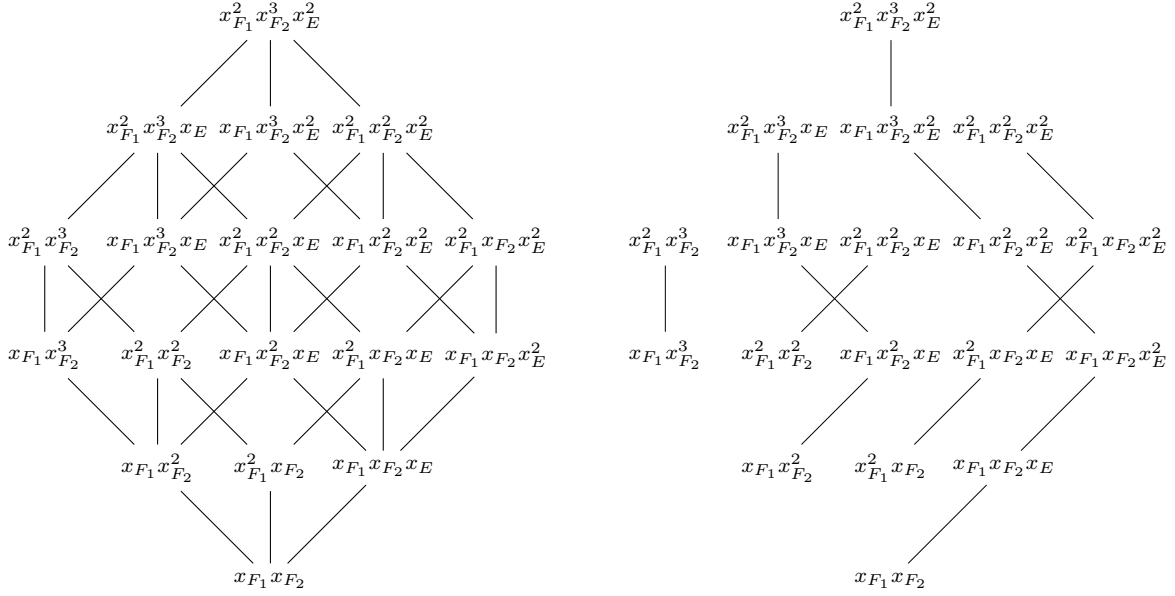
\begin{center}
\begin{tikzpicture}
  [scale=.25,auto=left]
   \node (110) at (-6,0) {\footnotesize $x_{F_1} x_{F_2}$};
  \node (120) at (-12,6) {\footnotesize $x_{F_1} x_{F_2}^2$};
  \node (210) at (-6,6) {\footnotesize $x_{F_1}^2 x_{F_2}$};
  \node (111) at (0,6) {\footnotesize $x_{F_1} x_{F_2} x_E$};
  \node (130) at (-18,12) {\footnotesize $x_{F_1} x_{F_2}^3$};
  \node (220) at (-12,12) {\footnotesize $x_{F_1}^2 x_{F_2}^2$};
  \node (121) at (-6,12) {\footnotesize $x_{F_1} x_{F_2}^2 x_E$};
  \node (211) at (0,12) {\footnotesize $x_{F_1}^2 x_{F_2} x_E$};
  \node (112) at (6,12) {\footnotesize $x_{F_1} x_{F_2} x_E^2$};
  \node (230) at (-18,18) {\footnotesize $x_{F_1}^2 x_{F_2}^3$};
  \node (131) at (-12,18) {\footnotesize $x_{F_1} x_{F_2}^3 x_E$};
  \node (221) at (-6,18) {\footnotesize $x_{F_1}^2 x_{F_2}^2 x_E$};
  \node (122) at (0,18) {\footnotesize $x_{F_1} x_{F_2}^2 x_E^2$};
  \node (212) at (6,18) {\footnotesize $x_{F_1}^2 x_{F_2} x_E^2$};
  \node (231) at (-12,24) {\footnotesize $x_{F_1}^2 x_{F_2}^3 x_E$};
  \node (132) at (-6,24) {\footnotesize $x_{F_1} x_{F_2}^3 x_E^2$};
  \node (222) at (0,24) {\footnotesize $x_{F_1}^2 x_{F_2}^2 x_E^2$};
   \node (232) at (-6,30) {\footnotesize $x_{F_1}^2 x_{F_2}^3 x_E^2$};
  \foreach \from/\to in {110/120,110/210,110/111,  120/130,120/220,120/121,210/220,210/211,111/211,111/121,111/112,
130/230,130/131,220/230,220/221,121/221,121/131,211/221,211/212,112/122,112/212,121/122,
230/231,131/231,221/231,131/132,221/222,122/132,212/222,122/222,
    231/232,132/232,222/232}
    \draw (\from) -- (\to);
  [scale=.25,auto=left,every node]
   \node (110) at (27,0) {\footnotesize $x_{F_1} x_{F_2}$};
  \node (120) at (21,6) {\footnotesize $x_{F_1} x_{F_2}^2$};
  \node (210) at (27,6) {\footnotesize $x_{F_1}^2 x_{F_2}$};
  \node (111) at (33,6) {\footnotesize $x_{F_1} x_{F_2} x_E$};
  \node (130) at (15,12) {\footnotesize $x_{F_1} x_{F_2}^3$};
  \node (220) at (21,12) {\footnotesize $x_{F_1}^2 x_{F_2}^2$};
  \node (121) at (27,12) {\footnotesize $x_{F_1} x_{F_2}^2 x_E$};
  \node (211) at (33,12) {\footnotesize $x_{F_1}^2 x_{F_2} x_E$};
  \node (112) at (39,12) {\footnotesize $x_{F_1} x_{F_2} x_E^2$};
  \node (230) at (15,18) {\footnotesize $x_{F_1}^2 x_{F_2}^3$};
  \node (131) at (21,18) {\footnotesize $x_{F_1} x_{F_2}^3 x_E$};
  \node (221) at (27,18) {\footnotesize $x_{F_1}^2 x_{F_2}^2 x_E$};
  \node (122) at (33,18) {\footnotesize $x_{F_1} x_{F_2}^2 x_E^2$};
  \node (212) at (39,18) {\footnotesize $x_{F_1}^2 x_{F_2} x_E^2$};
  \node (231) at (21,24) {\footnotesize $x_{F_1}^2 x_{F_2}^3 x_E$};
  \node (132) at (27,24) {\footnotesize $x_{F_1} x_{F_2}^3 x_E^2$};
  \node (222) at (33,24) {\footnotesize $x_{F_1}^2 x_{F_2}^2 x_E^2$};
   \node (232) at (27,30) {\footnotesize $x_{F_1}^2 x_{F_2}^3 x_E^2$};
  \foreach \from/\to in {110/111,
  120/121,210/211,111/112,
130/230,220/221,121/131,211/212,112/122,
131/231,122/132,212/222,
132/232}
    \draw[line width=0.1mm] (\from) -- (\to);
\end{tikzpicture} 
\end{center}
\end{example}

We can now define the maps $\pi, \lambda$ asserted in the theorem.
Let $a$ be any FY-monomial $a$ with $k:=\deg(a)$.  Then $a$ lies in a unique fiber 
$\supp_+^{-1}(N^+)$ for some $\G$-nested set 
$N^+=\{F_1,\ldots,F_\ell,E\}$
of the map $\supp_+$, and on a unique chain $C_i$
in our fixed SCD of this fiber.  Writing 
\begin{equation}
\label{typical-SCD-chain}
C_i=\{ a_{\rho} \lessdot a_{\rho+1} \lessdot \cdots \lessdot a_{r-\rho-1} \lessdot a_{r-\rho} \},
\end{equation}
where $\deg(a_j)=j$ for $j=\rho,\rho+1,\ldots,r-\rho$,
so that $a=a_k$, then define the bijection $\pi$ and maps $\lambda$ via
\begin{align}
\label{pi-def}
\pi(a)&:=a_{r-k},\\
\label{lambda-def}
\lambda(a)&:=a_{k+1} \text{ if }k < \frac{r}{2}.
\end{align}
It remains to show that $\pi,\lambda$ are $\subgroup$-equivariant. Note that having fixed $N^+=\{F_1,\ldots,F_\ell,F_{\ell+1}:=E\}$,
both maps $f=\pi, \lambda$ will
map a monomial of the form $a=\prod_{i=1}^{\ell+1} x_{F_i}^{m_i}$
with $m_i \geq 1$ for $1 \leq i \leq \ell$ as follows:
$$
a=\prod_{i=1}^{\ell+1} x_{F_i}^{m_i}
\quad
\longmapsto 
\quad
f(a)=\prod_{i=1}^{\ell+1} x_{F_i}^{m'_i}
$$
where the exponents $(m'_1,\ldots,m'_\ell,m'_{\ell+1})$
are determined in \eqref{pi-def}, \eqref{lambda-def} uniquely by
our choice of linear ordering $(F_1,\ldots,F_\ell)$ and
by the data 
$$
\left( 
\,\,
(m_1,\ldots,m_{\ell},m_{\ell+1}) 
\,\,
,
\,\,
(m_{N^+}(F_1),  m_{N^+}(F_2),
\ldots,
m_{N^+}(F_\ell),
m_{N^+}(F_{\ell+1}))
\,\,
\right).
$$
Recall from the proof of Corollary~\ref{Chow-ring-carries-perm-reps-cor} that $g$ in $\subgroup$
have $m_{g(N)}(g(F))=m_{N}(F)$.  Also, our $G$-equivariant choice of linear orderings dictates that $g(N^+)=\{g(F_1),\ldots,g(F_\ell),E\}$ has linear order $(g(F_1),\ldots,g(F_\ell))$.  Thus either map $f=\pi, \lambda$ sends
$$
g(a)=\prod_{i=1}^{\ell+1} x_{g(F_i)}^{m_i}
\quad
\longmapsto 
\quad
f(g(a))= \prod_{i=1}^{\ell+1} x_{g(F_i)}^{m'_i}
=g(f(a)). \qedhere
$$
\end{proof}

\begin{remark}\rm \label{rmk:conclusion-failure}
Theorem~\ref{main-theorem} can fail without the
assumption of the stabilizer condition \eqref{eq:stabilizer-condition}. To see that this is true, consider the graphic matroid associated to a complete graph on $12$ vertices, with lattice of flats $\Pi_{12}$ and minimal building set $\G_{\min}$ (see Example~\ref{partition-lattice-example}). For this building set, the stabilizer condition \eqref{eq:stabilizer-condition} does not hold. For example, we have the following nested set 
    \begin{align*}
    N = \{F_1, F_2, F_3\}, \text{ where } F_1 &= 1,2,3,4|5|6|7|8|9|10|11|12, \\
    F_2 &= 5,6,7,8|1|2|3|4|9|10|11|12, \\ 
    \text{ and } F_3 &= 9,10,11,12|1|2|3|4|5|6|7|8.
    \end{align*}
The set $N$ is nested because the join $F_1 \vee F_2 \vee F_3 = 1,2,3,4|5,6,7,8|9,10,11,12$ has multiple non-singleton blocks, and is therefore not a member of $\G_{\min}$. The permutation $g = (1,5,9)(2,6,10)(3,7,11)(4,8,12)$ in $\symm_{12}$, expressed here in cycle notation, fixes $N$ setwise, but not pointwise:
it sends $F_1 \mapsto F_2 \mapsto F_3 \mapsto F_1$.


Let us see how the conclusion of Theorem~\ref{main-theorem} also fails in this case. Here $N^+=\{F_1,F_2,F_3, E\}$ and each of the $F_i$ is of rank $3$ in $\Pi_{12}$, so $m_{N^+}(F_i)=3$. The join $F_1 \vee F_2 \vee F_3$ is of rank $9$, while the maximal element $E$ has rank $11$; thus we have $m_{N^+}(E)=2$. Following the proof of Theorem~\ref{main-theorem}, we have a poset isomorphism $\supp_+^{-1}(N^+) \cong C_2 \times C_2 \times C_2 \times C_2$:

\begin{tikzpicture} [scale=.7,auto=left]
  \node (0) at (0,0) {$x_{F_1} x_{F_2} x_{F_3}$};
  
  \node (f2gh) at (-6,2) {$x_{F_1}^2 x_{F_2} x_{F_3}$};
  \node (fg2h) at (-2,2) {$x_{F_1} x_{F_2}^2x_{F_3}$};
  \node (fgh2) at (2,2) {$x_{F_1} x_{F_2}x_{F_3}^2$};
  \node (fghe) at (6,2) {$x_{F_1} x_{F_2} x_{F_3}x_E$};
  
  \node (f2g2h) at (-10,4) {$x_{F_1}^2 x_{F_2}^2 x_{F_3}$};
  \node (f2gh2) at (-6,4) {$x_{F_1}^2 x_{F_2} x_{F_3}^2$};
  \node (fg2h2) at (-2,4) {$x_{F_1} x_{F_2}^2 x_{F_3}^2$};
  
  \node (f2ghe) at (2,4) {$x_{F_1}^2 x_{F_2} x_{F_3} x_E$};
  \node (fg2he) at (6,4) {$x_{F_1} x_{F_2}^2 x_{F_3} x_E$};
  \node (fgh2e) at (10,4) {$x_{F_1} x_{F_2} x_{F_3}^2 x_E$};

  \node (f2g2h2) at (-6,6) {$x_{F_1}^2 x_{F_2}^2 x_{F_3}^2$};  
  \node (f2g2he) at (-2,6) {$x_{F_1}^2 x_{F_2}^2 x_{F_3} x_E$};
  \node (f2gh2e) at (2,6) {$x_{F_1}^2 x_{F_2} x_{F_3}^2 x_E$};
  \node (fg2h2e) at (6,6) {$x_{F_1} x_{F_2}^2 x_{F_3}^2 x_E$};
  
  \node (1) at (0,8) {$x_{F_1}^2 x_{F_2}^2 x_{F_3}^2 x_E$}; 

  \foreach \from/\to in {0/f2gh, 0/fg2h, 0/fgh2, 0/fghe, 
  					f2gh/f2g2h, f2gh/f2gh2, f2gh/f2ghe, 
					fg2h/fg2h2, fg2h/f2g2h, fg2h/fg2he,
					fgh2/f2gh2, fgh2/fg2h2, fgh2/fgh2e,
					fghe/f2ghe, fghe/fg2he, fghe/fgh2e,
					f2g2h/f2g2h2, f2g2h/f2g2he,
					f2gh2/f2g2h2, f2gh2/f2gh2e,
					fg2h2/f2g2h2, fg2h2/fg2h2e,
					f2ghe/f2g2he, f2ghe/f2gh2e,
					fg2he/f2g2he, fg2he/fg2h2e,
					fgh2e/f2gh2e, fgh2e/fg2h2e,
					f2g2h2/1, f2g2he/1, f2gh2e/1, fg2h2e/1}
 		\draw (\from) -- (\to);  
\end{tikzpicture}

If Theorem~\ref{main-theorem} held in this case, there would be an injective $G$-equivariant map $\lambda: \FY^4 \to \FY^5$. Injective $G$-equivariant maps preserve stabilizer subgroups; in other words, for any $m \in \FY$, one has $g \cdot m = m$ if and only if $g \cdot \lambda(m) = \lambda(m)$. We explain here why in this particular case, that requirement cannot be met.

The monomial $x_{F_1} x_{F_2} x_{F_3}x_E$ has a stabilizer subgroup isomorphic to $\symm_3[\symm_4]$, the wreath product of $\symm_3$ and $\symm_4$. Concretely, this stabilizer subgroup is generated by permutations stabilizing each of $\{1,2,3,4\}$, $\{5,6,7,8\}$, and $\{9,10,11,12\}$ setwise, and also permutations which swap these sets `wholesale.'  One can check that all monomials in $\FY(\Pi_{12})$ which have this stabilizer lie in the above fiber $\supp^{-1}(N^+)$
of the extended support map. Thus when attempting to construct $\lambda$, the image $\lambda(x_{F_1} x_{F_2} x_{F_3}x_E)$ must lie among the degree $5$ monomials in $\supp_+^{-1}(N^+)$. However, none of these degree $5$ monomials have the same stabilizer subgroup $\symm_3[\symm_4]$ as $x_{F_1} x_{F_2} x_{F_3}x_E$. For example, the permutation $g = (1,5,9)(2,6,10)(3,7,11)(4,8,12)$ mentioned above fixes
$x_{F_1} x_{F_2} x_{F_3}x_E$, but none of the degree $5$ monomials.
\end{remark}

\begin{remark} \rm
The $\subgroup$-equivariant bijections 
$
\pi: \FY^k  \overset{\sim}{\longrightarrow}  \FY^{r-k}
$
for $k \leq \frac{r}{2}$
in the above proof have an extra property:  a monomial $a$ always divides its image $\pi(a)$.  If one does not insist on this property, then
one has the following simpler $\subgroup$-equivariant bijection: 
given
$
a=x_{F_1}^{m_1} x_{F_2}^{m_2} \cdots x_{F_\ell}^{m_\ell} x_E^{m_{\ell+1}}
$
in $\FY^k$ 
where $k \leq \frac{r}{2}$
lying in
$
\supp_+^{-1}(N^+)\{F_1,\ldots,F_\ell,E\},
$
where $N^+=
\{F_1,\ldots,F_\ell,E\}
$
so that the $m_i$ satisfy the inequalities \eqref{fiber-inequalities}, then map the FY-monomial $a$ to this 
FY-monomial in the same fiber:
$$
a'=x_{F_1}^{m_{N^+}(F_1)-m_1} x_{F_2}^{m_{N^+}(F_2)-m_2}  \cdots
x_{F_\ell}^{m_{N^+}(F_\ell)-m_\ell} 
x_E^{m_{N^+}(E)-1-m_{\ell+1}}.
$$
The authors thank Connor Simpson for pointing out that the latter bijection is (up to a plus/minus sign) the specialization of a bijection appearing in the work of Pagaria and Pezzoli \cite[Defn.~4.3]{PagariaPezzoli}, where they produce an explicit Poincar\'e duality isomorphism for Chow rings of all {\it polymatroids}.
\end{remark}

In the remainder of this section, we explain
a combinatorial proof of Theorem \ref{main-theorem} that works only for the maximal building set $\G_{\max}$; it can be seen as a way of making concrete choices of the SCDs in the previous proof. Recall that the \textit{maximal} building set is $\G_\text{max} = \L \setminus \{\emptyset\}$. The nested sets in $\G_\text{max}$ are simply chains in $\L \setminus \{\hat{0}\}$ (see Section~\ref{building-set-basics}), and the Feichtner-Yuzvinsky monomials with respect to $\G_\text{max}$ have the form
    \[\fy := \left\{ x_{F_1}^{m_1} \cdot x_{F_2}^{m_2}\cdots x_{F_\ell}^{m_\ell}: \emptyset = F_0 \subsetneq F_1 \subsetneq \ldots \subsetneq F_\ell \text{ is a chain and } m_i < \rk(F_i) - \rk(F_{i-1}) \text{ for all } i \right\}.\]
This proof borrows an idea from the famous {\it parenthesis-pairing}
SCD of Boolean lattices due to Greene and Kleitman \cite{GK}.  We begin with a two-step encoding of FY-monomials within a fiber of the map $\supp_+$. 

\begin{defn} \rm
Consider all FY-monomials $a=x_{F_1}^{m_1} \cdots x_{F_\ell}^{m_\ell} x_E^{m_{\ell+1}}$ having a fixed extended support set
$\supp_+(a)=\{F_1 \subsetneq \cdots \subsetneq F_\ell \subsetneq F_{\ell+1}=E\}$, so $m_1,\ldots,m_\ell \geq 1$ and $m_{\ell+1} \geq 0$.  In the first step, encode such monomials
$a$ via
a sequence $\DD(a)$ of length $r$
in three symbols $\times, \bullet$, and a blank space, defined as follows:
\begin{itemize}
    \item $\DD(a)$ has $\bullet$ in the positions $\{\rk(F_1),\ldots,\rk(F_\ell)\}$.
\item  $\DD(a)$ has $\times$ in the first consecutive $m_i$ positions to the left of $\rk(F_i)$ for each $i=1,2,\ldots,\ell,\ell+1$.
\item $\DD(a)$ has a blank space in the remaining positions.
\end{itemize}

\begin{example} \label{dd-example}
\rm Continuing with the matroid $\M$ and its flats $F_1 < F_2$ as discussed in Example~\ref{extended-support-map-example}.  Here the monomials lie in the fiber
$\supp_+^{-1}\{F_1,F_2,E\}$ where 
$\{\rk(F_1),\rk(F_2),\rk(E)\}=\{3,7,10\}$, so $r=9$, and the positions $\{\rk(F_1),\rk(F_2)\}=\{3,7\}$ are shown in {\color{teal} teal}. The monomial $x_{F_1} x_{F_2}^2$ gets encoded as
\[\begin{array}{ccccccccc}
     1 & 2 & {\color{teal} 3} & 4 & 5 & 6 & {\color{teal} 7} & 8 & 9  \\ \\
      & \times & {\color{teal} \bullet} &  & \times & \times & {\color{teal} \bullet} &  & 
\end{array}\]

\end{example}

\end{defn}

\noindent
Note that one can recover $a$ from 
$
\supp_+(a)=\{F_1,\ldots,F_\ell, E\}
$
and $\DD(a)$, since $m_i$ can be read off as the number of $\times$ in $\DD(a)$ between positions $\rk(F_{i-1})$ and $\rk(F_i)$, with usual conventions
$F_0=\varnothing, F_{\ell+1}:=E$.

\begin{defn} \rm
The second step encodes $\DD(a)$ as a length $r$ parenthesis sequence in $\{(,)\}^r$,  having \begin{itemize}
    \item a right parenthesis ``$)$" in the positions of each $\bullet$ and each blank space, and 
    \item a left parenthesis ``(" in the positions of the $\times$.
\end{itemize}

\end{defn}

\begin{example} \label{dd-to-()-example}
  \rm  Continuing Example~\ref{dd-example}, the diagram $\DD(x_{F_1} x_{F_2}^2)$ is encoded as this sequence of parentheses:
    \[\begin{array}{ccccccccc}
     1 & 2 & {\color{teal} 3} & 4 & 5 & 6 & {\color{teal} 7} & 8 & 9  \\ \\
      & \times & {\color{teal} \bullet} &  & \times & \times & {\color{teal} \bullet} &  & \\
     ) & ( & ) & ) & ( & ( & ) & ) & )
\end{array}\]
\end{example}

\noindent
Note that one can recover $\DD(a)$ from this
 $\{(,)\}^r$ sequence as follows: 
 \begin{itemize}
  \item[-]  $\DD(a)$ has $\times$ occurring in the positions of the left parentheses, and 
     \item[-] $\DD(a)$ has the $\bullet$ occurring
 exactly in the positions of the right parenthesis in a consecutive pair $()$, while blank spaces occur in the position of all other right parentheses.
 \end{itemize}

\begin{proof}[Second proof of Theorem~\ref{main-theorem} for $\G=\G_{\max}$.]

Given an FY-monomial $a$ in $\FY^k$ with $k \leq \frac{r}{2}$, we will use its parenthesis sequence in $\{(,)\}^r$ to place $a$ within a chain of monomials as in \eqref{typical-SCD-chain}, of
the form
$$
C_i=\{ a_{\rho} \lessdot a_{\rho+1} \lessdot \cdots \lessdot a_{r-\rho-1} \lessdot a_{r-\rho} \},
$$
where $\deg(a_j)=j$ for $j=\rho,\rho+1,\ldots,r-\rho$,
so that $a=a_k$.  To this end, define the set of {\it paired parentheses} in $a$ (shown underlined in Example~\ref{example-re-revisited})
by including all consecutive pairs $()$, and after removing these
consecutive pairs, including the new $()$ pairs which have become consecutive, repeating this recursively.  After removing some number $\rho$ of pairs $()$ via this pairing process, the process ends when one reaches a sequence of $r -2 \rho$ remaining unpaired parentheses of this form:
\begin{equation}
\label{unpaired-parentheses}
\underbrace{)) \cdots ))}_{k-\rho}
\underbrace{{\color{orange}(( \cdots ((}}_{r-\rho-k}.
\end{equation}
The monomials in the chain $C_i$ are defined to
be those whose set of paired parentheses agree
exactly with those of $a$, both in their postitions, and left-right pairing structure-- see the underlined parentheses in Example~\ref{example-re-revisited}.  

As in the first proof of Theorem~\ref{main-theorem}, one then defines the bijections $\pi$ and maps $\lambda$ via
$$
\begin{aligned}
\pi(a)&:=a_{r-k},\\
\lambda(a)&:=a_{k+1}  \text{ if }k < \frac{r}{2}.
\end{aligned}
$$
In other words, both maps $\lambda$ and $\pi$
when applied to $a$ will keep all of the paired parentheses fixed, but 
\begin{itemize}
    \item $\lambda$ changes the leftmost unpaired left parenthesis ``)" into an unpaired right parenthesis ``{\color{orange}(}", and
    \item $\pi$ swaps the numbers $k-\rho$ and $r-\rho-k$ of unpaired right and left parentheses in \eqref{unpaired-parentheses}.
\end{itemize}

The equivariance of these two maps $\pi, \lambda$ is argued exactly as in the first proof of the theorem.
\end{proof}

\begin{example}\rm
\label{example-re-revisited}
Here is an example of a symmetric chain from Example~\ref{extended-support-map-example},
explained via this two-step encoding with $\DD(a)$ and
$\{(,)\}^r$-sequences, as we have seen preceding Examples~\ref{dd-example} and ~\ref{dd-to-()-example}. 
Paired parentheses are underlined, and are fixed throughout the chain. Moving up the chain, unpaired right parentheses change one-by-one to left parentheses (depicted {\color{orange} orange} here), in order from right to left:

\small
$$
\begin{array}{ccccccccccl}
 & 1 & 2 & {\color{teal}3} & 4 & 5 & 6 & {\color{teal}7} & 8 & 9 \\
    &  &  &  &  &  &  &  &  &  \\
x_{F_1}^2 x_{F_2}^3 x_E & \times & \times & {\color{teal}\bullet} & \times & \times & \times &{\color{teal}\bullet}  &  & \times\\
& {\color{orange}(} & \underline{(} & \underline{)} & {\color{orange}(} & \underline{(} & \underline{(} & \underline{)} & \underline{)} & {\color{orange}(}\\ 
    &  &  &  &  &  &  &  &  &  \\
 \uparrow &  &  &  &  &  &  &  &  &  \\
    &  &  &  &  &  &  &  &  &  \\
x_{F_1} x_{F_2}^3 x_E &  & \times & {\color{teal}\bullet} & \times & \times & \times &{\color{teal}\bullet}  &  & \times\\
& ) & \underline{(} & \underline{)} & {\color{orange}(} & \underline{(} & \underline{(} & \underline{)} & \underline{)} & {\color{orange}(}\\ 
    &  &  &  &  &  &  &  &  &  \\
 \uparrow &  &  &  &  &  &  &  &  &  \\
    &  &  &  &  &  &  &  &  &  \\
x_{F_1} x_{F_2}^2 x_E &  & \times & {\color{teal}\bullet} &  & \times & \times &{\color{teal}\bullet}  &  & \times\\
& ) & \underline{(} & \underline{)} & ) & \underline{(} & \underline{(} & \underline{)} & \underline{)} & {\color{orange}(}\\ 
   &  &  &  &  &  &  &  &  &  \\
 \uparrow &  &  &  &  &  &  &  &  &  \\
    &  &  &  &  &  &  &  &  &  \\
x_{F_1} x_{F_2}^2 
 &  & \times & {\color{teal}\bullet} &   & \times & \times &{\color{teal}\bullet}  &  &  \\
 & ) & \underline{(} & \underline{)} & ) & \underline{(} & \underline{(} & \underline{)} & \underline{)} & )\\ 
\end{array}
$$
\normalsize

Here $r=9$ and the number of parenthesis pairs is $\rho=3$, so that this chain $C_i$ is symmetrically placed within the degrees of $A(\L_\M, \G_{\max})$, containing monomials of degrees $[\rho,r-\rho]=[3,6]=\{3,4,5,6\}$ 
out of the list of possible degrees $[0,r]=[0,9]=\{0,1,2,\mathbf{3,4,5,6},7,8,9\}$. 
\end{example}

\section{Conjectures on equivariant and Burnside ring inequalities}
\label{conjectures-section}
\label{rep-theoreric-inequalities-section}

Recall from the Introduction that the unimodality statement \eqref{unimodality}, asserting
for $k < \frac{r}{2}$ the inequality
$$
a_k \leq a_{k+1},
$$
is weaker than the statement in Corollary~\ref{AHK-equivariant-Hard-Lefschetz} asserting that there are injective $\R \subgroup$-module maps
$$
A^k_\R \hookrightarrow A_\R^{k+1},
$$
which is weaker than 
Theorem~\ref{main-theorem}(ii) asserting that there are
injective $\subgroup$-equivariant maps of the $\subgroup$-sets
$$
\FY^k \hookrightarrow \FY^{k+1}.
$$
We also introduced there the ring $R_\C(G)$ of virtual complex $G$-characters and the Burnside ring $B(G)$ of virtual $G$-sets,
along with maps $B(G) \rightarrow R_\C(G)$ in
\eqref{Burnside-to-character-map} and a map 
$R_\C(G) \rightarrow \Z$ in \eqref{characters-to-integers-map}. This allowed us to view the inequality $a_k \leq a_{k+1}$ as lifting through the map $R_\C(G) \rightarrow \Z$ to an inequality \eqref{character-unimodality} of the form $
\chi_{A^k_\R} \leq_{R_\C(G)} \chi_{A_\R^{k+1}},
$
in $R_\C(G)$, which then lifts 
through the map $B(G) \rightarrow R_\C(G)$ to an inequality
of the form \eqref{Burnside-unimodality} 
of the form $[\FY^k] \leq_{B(G)} [\FY^{k+1}]$ in $B(G)$.  The goal of this section is to go beyond unimodality for $(a_0,a_1,\ldots,a_r)$, 
considering properties
like {\it log-concavity}, {\it the P\'olya frequency property}, and similarly lifting them to statements in $R_\C(G)$ and $B(G)$.

In the process, we will make frequent use of the following fact:  just as one can multiply inequalities in $\Z$ like $a< b$ and $c < d$ to get new inequalities $ac < bc < bd$, the same works in $R_\C(\subgroup)$
and in $B(\subgroup)$.  This is because
 $\chi, \chi' \geq_{R_\C(\subgroup)} 0$ implies $\chi \cdot \chi' \geq_{R_\C(\subgroup)} 0$, and similarly
$b, b' \geq_{B(\subgroup)} 0$ implies $b \cdot b' \geq_{B(\subgroup)} 0$.

\subsection{PF sequences and log-concavity}
For a sequence of {\it positive} real numbers $(a_0,a_1,\ldots,a_r)$, the property of {\it unimodality} lies at the bottom of a hierarchy of concepts
\begin{equation}
\label{PF-hierarchy}
\begin{array}{rccccccccccccccc}
\text{unimodal}
&\Leftarrow&
PF_2
&\Leftarrow&
&PF_3&
&\Leftarrow&
&PF_4&
&\Leftarrow&
\cdots&
\Leftarrow&
&PF_\infty\\
&  & \Vert & & & &
& & & & & & & &
&\Vert\\
&  & \text{(strongly) log-concave} & & & &
& & & & & & & &
&PF\\
\end{array}
\end{equation}
which we next review, along with their equivariant and Burnside ring extensions.  For background on the non-equivariant versions, see Brenti \cite{Brenti} and Stanley \cite{Stanley-log-concavity}.  For the equivariant versions, see Gedeon, Proudfoot and Young \cite{GPY},
Matherne, Miyata, Proudfoot and Ramos \cite{MMPR}, Gui \cite{Gui2022}, Gui and Xiong \cite{GuiXiong}, and Li \cite{Li2022}.

\begin{defn}\rm
\label{numerical-inequality-conditions}
    Say a sequence of positive reals $(a_0,a_1,\ldots,a_r)$ is {\it unimodal} if there is some index $m$ with 
    $$
    a_0 \leq a_1 \leq \cdots \leq a_{m-1} \leq a_m \geq a_{m+1} \geq \cdots \geq a_{r-1} \geq a_r.
    $$
    
    Say the sequence is {\it strongly\footnote{The word ``strongly" here is superfluous, since we assumed each $a_k >0$, so they are strongly log-concave if and only they are weakly log-concave:  $a_k^2 \geq a_{k-1} a_{k+1}$.  The distinction becomes important for the equivariant analogue; see \cite[\S 2]{MMPR}.} log-concave} (or {\it $PF_2$}) if  $0 \leq i\leq j \leq k \leq \ell\leq r$ and $i+\ell=j+k$ implies
    $$
    a_i a_\ell \leq a_j a_k, \text{ or equivalently,  }
    \det\left[ \begin{matrix} a_j & a_\ell \\ a_i & a_k \end{matrix}\right] \geq 0.
    $$
    
    For $\ell=2,3,4,\ldots$, say that the sequence is $PF_\ell$ if the associated (infinite) {\it Toeplitz matrix} 
    $$
    T(a_0,\ldots,a_r):=
    \left[
    \begin{matrix}
        a_0 & a_1& a_2& \cdots & a_{r-1} & a_r     & 0   & 0 & \cdots  \\
        0   & a_0 & a_1 & \cdots & a_{r-2} & a_{r-1} & a_{r} & 0 &\cdots \\
        0   & 0   & a_0 & \cdots & a_{r-3} & a_{r-2} & a_{r-1} & a_{r}& \cdots \\
        \vdots & \vdots & \vdots & \vdots & \vdots & \vdots & \vdots & \vdots & \ddots 
    \end{matrix}
    \right]
    $$
    has all {\it nonnegative} square minor subdeterminants of size 
    $m \times m$ for $1\leq m  \leq \ell$.
    Say that the sequence is a {\it P\'olya frequency sequence} (or {\it $PF_\infty$}, or just {\it $PF$}) if it is $PF_\ell$ for all $\ell=2,3,\ldots$.
\end{defn}
\noindent
One can check the implication ($PF_2$ implies
unimodality) from \eqref{PF-hierarchy} using the assumption 
$a_k >0$ for all $k$.
It also turns out that $(a_0,a_1,\ldots,a_r)$ is $PF$ if and only if the polynomial $a_0 + a_1 t+a_2 t^2+ \cdots+ a_r t^r$ has only (negative) real roots; see \cite[\S 2.2, 4.5]{Brenti}.

\begin{defn} \rm
For a finite group $\subgroup$ and (genuine, nonzero) $\C \subgroup$-modules $(A^0,A^1,\ldots,A^r)$, define the analogous
notions of {\it equivariant unimodality}, {\it equivariant strong log-concavity}, {\it equivariant $PF_r$ or $PF_\infty$} by replacing the numerical inequalities in 
Definition~\ref{numerical-inequality-conditions} by
inequalities in the representation ring $R_\C(\subgroup)$. 

Similarly, for (nonempty) $\subgroup$-sets
$(X_0,X_1,\ldots,X_r)$, define the notions of  {\it Burnside unimodality}, {\it Burnside strong log-concavity}, {\it Burnside $PF_r$ or $PF_\infty$} by replacing them with  inequalities in the Burnside ring $B(\subgroup)$. 
\end{defn}

\begin{example} \rm
\label{unimodality-of-various-kinds-example}
We've seen the following holds for Chow rings $A(\M, \G)=\bigoplus_{k=0}^r A^k$ of rank $r+1$ matroids $\M$ with $\subgroup=\Aut(\M)$, and choice of $\subgroup$-stable building set $\G$ containing $E$ that satisfies \eqref{eq:stabilizer-condition}:
\begin{itemize}
    \item 
the sequence $(a_0,a_1,\ldots,a_r)$ with $a_k:=\rk_\Z A_k$ is {\it unimodal}, 
\item after tensoring with $\C$, the sequence of $\C \subgroup$-modules $(A^0_\C,A^1_\C,\ldots,A^r_\C)$ is {\it equivariantly unimodal}, and
\item the sequence of $\subgroup$-sets $(\FY^0,\FY^1,\ldots,\FY^r)$
is {\it Burnside unimodal}.
\end{itemize}
\end{example}

 We give here several conjectures for the Chow ring $A(\L_\M,\G_{\max})$ of a simple matroid $\M$ with the choice of the maximal building set $\G_{\max}$.

\begin{conj}
\label{log-concavity-conjectures}
For any simple matroid $\M$ with lattice of flats $\L=\L_\M$ of rank $r+1$ and $\subgroup:=\Aut(\M)$, the Chow ring $A(\L,\G_{\max})=\oplus_{k=}^r A_k$
and the sequence $(a_0,a_1,\ldots,a_r)$ with $a_k:=\mathrm{rank}_\Z A_k$
satisfy the following:
    \begin{itemize}
    \item[(i)] (Ferroni-Schr\"oter \cite[Conj. 10.19]{FerroniSchroter}) $(a_0,\ldots,a_r)$ is $PF_\infty$.
\item[(ii)] $(A^0_\C,\ldots,A^r_\C)$ is equivariantly $PF_\infty$.
\item[(iii)] $(\FY^0,\ldots,\FY^r)$
is Burnside $PF_2$ (Burnside log-concave), that is,
\[
[\FY^i][\FY^\ell] \leq_{B(\subgroup)}
[\FY^j][\FY^k]
\qquad \text{for } 
 \; i \leq j \leq k \leq \ell \; \text{ with } \; i+\ell=j+k.\]
\end{itemize}
\end{conj}
Of course, in Conjecture~\ref{log-concavity-conjectures}, 
assertion (ii) implies assertion (i).
However assertion (iii) would only imply the weaker $PF_2$ part of 
the conjectural assertion (ii), and only imply the $PF_2$ part of Ferroni and Schr\"oter's assertion (i), but not their $PF_\infty$ assertions. Even the $PF_2$ property for $(a_0,\ldots,a_r)$ is still conjectural; see \cite[\S10.3]{FerroniSchroter} and \cite[\S 3.7]{FMSV} for a discussion 
of the current evidence for Conjecture~\ref{log-concavity-conjectures}(i).

\begin{example}
\label{Burnside-PF3-counterexample}
 \rm \ 
We explain here why Conjecture~\ref{log-concavity-conjectures} {\it does not} assert that $(\FY^0,\ldots,\FY^r)$ is Burnside $PF_\infty$.  In fact, $(\FY^0,\ldots,\FY^r)$ {\it fails even to be Burnside $PF_3$}, already when $\M$ is a rank $4$ Boolean matroid. Its Chow ring $A(\L_\M,\G_{\max})=A^0 \oplus A^1 \oplus A^2 \oplus A^3$ has $A_0,A_3$ carrying the trivial $\C \symm_4$-module, and $A^1,A^2$ carrying isomorphic permutation representations, each having three orbits, whose three $\symm_4$-stabilizer groups are the Young subgroups $\symm_4, \symm_3 \times \symm_1, \symm_2 \times \symm_2$.  
The red $3 \times 3$ minor 
of the Toeplitz matrix shown here
$$
\left[
\begin{matrix}
    a_0 & {\color{red}a_1} & {\color{red}a_2} & a_3 & {\color{red} 0} & 0 &\cdots \\
       0  & {\color{red}a_0} & {\color{red}a_1}& a_2 & {\color{red} a_3} & 0 & \cdots \\
       0  &  {\color{red} 0}   & {\color{red} a_0} & a_1 & {\color{red}a_2} & a_3 &\cdots \\
        0  &  0 & 0   &  a_0 & a_1 & a_2  &\cdots \\
  \vdots  &  \vdots   & \vdots & \vdots & \vdots & \vdots & \ddots \\
\end{matrix}
\right]
$$
has determinant 
$
{\color{red}a_1^2 a_2-a_1 a_3 - a_2^2}.
$
Hence the Burnside $PF_3$ condition would require that the following genuine $\symm_4$-character should come from a genuine permutation representation
$$
{\color{red}\left( \chi_{A^1}\right)^2 \cdot \chi_{A^2} - 
\chi_{A^1} \chi_{A^3} - \left(\chi_{A^2}\right)^2
}= 29 \chi^{(1, 1, 1, 1)} + 124 \chi^{(2, 1, 1)} + 
103 \chi^{(2, 2)} + 172 \chi^{(3, 1)} + 76\chi^{(4)},
$$
where here $\chi^\lambda$ denotes the irreducible $\symm_n$-representation \cite{Sagan}, \cite[\S 7.18]{Stanley-EC2} indexed by the partition $\lambda$ of $n$; this expansion was computed using {\tt Sage/Cocalc}.  But one can check that this is {\it not} a permutation representation, as its character value on the conjugacy class of $4$-cycles in $\symm_4$ is $76-172+124-29=-1<0$.
\end{example}

\begin{remark} \rm \
\label{PF-needs-max-building-set}
There would exist small counterexamples if one allowed $\G \neq \G_{\max}$ in Conjecture~\ref{log-concavity-conjectures}(i),(ii).  For example, 
Maestroni and McCullough point out in
\cite[Ex.~6.3]{MaestroniMcCullough} that the uniform matroid $\M=U_{3,4}$ of rank $3$ on $4$ elements has
$A(\L_\M,\G_{\min}) \cong \Z[x_E]/(x_E^3)$ with ranks $(a_0,a_1,a_2)=(1,1,1)$, a sequence which is not even $PF_3$.  On the other hand,
we have not extensively investigated Conjecture~\ref{log-concavity-conjectures}(iii) whether might hold when generalized to all $\subgroup$-stable
building sets $\G$ containing $E$.
\end{remark}

\begin{remark} \rm  \ 
Although Example~\ref{Burnside-PF3-counterexample} shows that even Boolean lattices/matroids contradict strengthening
Conjecture~\ref{log-concavity-conjectures}(iii) to Burnside $PF_3$, they seem to satisfy a {\it different} strengthening of strong log-concavity:

\begin{conj}
\label{boolean-h-log-concavity-conj}
Consider the Boolean lattice $\L$ of rank $n$ with $\G=\G_{\max}$, and $\subgroup=\symm_n$ the symmetric group.  Then for any $i \leq j \leq k \leq \ell$ with  $i+\ell=j+k$, not only is the element
$$
[\FY^j][\FY^k]-[\FY^i][\FY^\ell] \geq_{B(\symm_n)} 0,
$$
but in fact a permutation representation with 
orbit-stabilizers all Young subgroups
$\symm_\lambda:=\symm_{\lambda_1} \times \symm_{\lambda_2} \times \cdots \times \symm_{\lambda_\ell}$.
\end{conj}
\end{remark}

We note here a small amount of evidence for Conjecture~\ref{log-concavity-conjectures}(ii), (iii), namely their strong log-concavity assertions
hold for the case
$i=0$. For assertion (ii), this is an easy consequence
of the fact that the Chow ring $A(\L,\G_{\max})$ is generated by the variables $\{ y_F \}$ spanning its degree one component $A^1$, which shows that this $\subgroup$-equivariant multiplication map surjects:
\begin{equation}
\label{multiplication-surjects}
A^j \otimes A^k  \twoheadrightarrow  A^{j+k} \left(\cong A^0 \otimes A^{j+k}\right).
\end{equation}
We next check that the
stronger assertion of Conjecture~\ref{log-concavity-conjectures}(iii) also holds in the special case $i=0$.

\begin{prop}
\label{thm:2by2kos}
Any simple matroid has $\subgroup$-equivariant injection 
$
\FY^{j+k} \hookrightarrow \FY^j \times \FY^k
$
for $j,k \geq 0$.
\end{prop}
\begin{proof}
Given $a=x_{F_1}^{m_1} \cdots x_{F_\ell}^{m_\ell}$ in $\FY^{j+k}$, so that $j+k=\sum_{i=1}^\ell m_i$, let $p$ be the smallest 
index such that
\begin{equation}
\label{index-definining-inqualities}
\sum_{i=1}^{p-1} m_i < j \leq \sum_{i=1}^{p} m_i
\end{equation}
and factor the monomial $a=b \cdot c$ where
$$
a= \underbrace{x_{F_1}^{m_1} \cdots x_{F_{p-1}}^{m_{p-1}} x_{F_p}^\delta}_{b}
\quad \cdot \quad 
\underbrace{ x_{F_p}^{m_p-\delta} x_{F_{p+1}}^{m_{p+1}}  \cdots x_{F_\ell}^{m_\ell}}_{c}
$$
with 
$\delta:=j-\sum_{i=1}^{p-1}m_i$ 
(so $\delta > 0$ by \eqref{index-definining-inqualities}), and
$m_p-\delta \geq 0$.  One can check that, since $a$ lies in $\FY^{j+k}$,
one will also have $b, c$ lying in $\FY^j, \FY^k$, respectively.
It is easily seen that the map $a \longmapsto (b,c)$ is injective,
since its inverse sends $(b,c) \longmapsto bc$. It is also not
hard to check that it is $\subgroup$-equivariant.\qedhere
\end{proof}

\begin{remark} \rm
Note that one can iterate the map in the previous proof to construct $\subgroup$-equivariant 
injections 
$
\prod_{i=1}^q \FY^{\alpha_i} \into \prod_{j=1}^p \FY^{\beta_j}
$
whenever $\beta = (\beta_1, \beta_2, \ldots, \beta_p)$ is a composition refining $\alpha = (\alpha_1, \alpha_2, \ldots, \alpha_q).$ 
\end{remark}

As another small piece of evidence for Conjecture \ref{log-concavity-conjectures} (ii), (iii), we show that $(\FY^0, \ldots, \FY^r)$ is Burnside $PF_2$ for $r \leq 5$, that is, for matroids of rank at most $6$.

\begin{prop} \label{thm:small-burnside-pf2}
For simple matroids $\M$ with $\rk(\M) \leq 6$, the sequence $(\FY^0, \ldots, \FY^r)$ is Burnside $PF_2$.
\end{prop}

\begin{proof}[Proof sketch.]
We check it for $\rk(\M)=6$, and $\rk(\M) \leq 5$ is similar. Theorem~\ref{main-theorem}(i) shows  
that in $B(\subgroup)$,
\[
 \Big([\FY^0], \, [\FY^1], \, [\FY^2], \, [\FY^3], \,[\FY^4], \, [\FY^5] \Big)
=
\Big(1, \, [\FY^1], \,[\FY^2], \, [\FY^2], \,[\FY^1], \, 1\Big).
\]

Hence one must check nonnegativity in $B(\subgroup)$
for all 2 x 2 minors in this infinite Toeplitz matrix:
\[
\begin{bmatrix}
    1 & [\FY^1] & [\FY^2] & [\FY^2] & [\FY^1] & 1 & 0 & 0 & \ldots \\
    0 & 1 & [\FY^1] & [\FY^2] & [\FY^2] & [\FY^1] & 1 & 0 & \ldots \\
    0 & 0 & 1 & [\FY^1] & [\FY^2] & [\FY^2] & [\FY^1] & 1 & \ldots \\
    0 & 0 & 0 & 1 & [\FY^1] & [\FY^2] & [\FY^2] & [\FY^1] & \ldots \\
    0 & 0 & 0 & 0 & 1& [\FY^1] & [\FY^2] & [\FY^2] & \ldots \\
    \vdots & \vdots & \vdots & \vdots & \vdots & \vdots & \vdots & \vdots & \ddots \\
\end{bmatrix}
\]
From periodicity of the matrix, one may assume without loss of generality that the $2 \times 2$ minor has its top-left entry in the first row.  If the minor has a $0$ as either its lower-left or upper-right entry, then its determinant is a product $[\FY^i] [\FY^j] = [\FY^i \times \FY^j] \geq_{B(\subgroup)} 0$. This already leaves only finitely many $2 \times 2$ minors to check.  Additionally, if it has $1$ as its lower left entry, then it was shown to be Burnside-nonnegative
in Theorem~\ref{thm:2by2kos}.  All of the remaining $2 \times 2$ minors we claim 
are Burnside-nonnegative because they compare two (possibly non-consecutive) terms in this chain of inequalities:
$$
1
\overset{(a)}{\leq}_{B(\subgroup)}[\FY^1] 
\overset{(b)}{\leq}_{B(\subgroup)} [\FY^2] 
\overset{(c)}{\leq}_{B(\subgroup)} [\FY^1][\FY^1] 
\overset{(d)}{\leq}_{B(\subgroup)} [\FY^1][\FY^2] 
\overset{(e)}{\leq}_{B(\subgroup)} [\FY^2][\FY^2]
$$
where inequalities (a),(b) follow from Theorem~\ref{main-theorem}(ii), inequality (c) follows from Theorem~\ref{thm:2by2kos}, and inequality (d),(e) come from multiplying inequality (b) by $[\FY^1]$ and multiplying inequality (a) by $[\FY^2]$.
\end{proof}

\begin{remark} \rm
When $\rk(\M) \geq 7$, one encounters the first $2 \times 2$ minor in $B(\subgroup)$ for $\subgroup=\Aut(\M)$
$$
\det 
\left[
\begin{matrix}
    [\FY^2] & [\FY^3] \\
    [\FY^1] & [\FY^2]
\end{matrix}
\right]
=
[\FY^2] [\FY^2] - [\FY^1] [\FY^3]
=[ \FY^2 \times \FY^2 ] - [\FY^1 \times \FY^3]
$$
whose Burnside nonnegativity does not already follow from our previous results.
\end{remark}

\subsection{Koszulity}

The surjection in \eqref{multiplication-surjects} that proved a special case of Conjecture~\ref{log-concavity-conjectures}(ii) turns out to be the $2 \times 2$ special case of more general equivariant $\ell \times \ell$ Toeplitz minor inequalities for Chow rings $A(\L_\M,\G_{\max})$.
These inequalities follow from general theory of
{\it Koszul algebras}, along with a recent result of Maestroni and McCullough \cite{MaestroniMcCullough} showing $A(\L_\M,\G_{\max})$ is Koszul.  After reviewing these
results, we state a conjecture generalizing Proposition~\ref{thm:2by2kos} and upgrade these Toeplitz minor inequalities from the representation ring $R_\C(\subgroup)$ to the Burnside ring $B(\subgroup)$.  As a reference for Koszul algebras, see
Polishchuk and Positselski \cite{PP}.

\begin{remark}
    The choice of the maximal building set $\G_{\max}$ for Koszulity is important. Maestroni and McCullough  \cite[Ex.~6.2, 6.3]{MaestroniMcCullough} exhibit small matroids $\M$ with $A(\L_\M,\G)$
    {\it not} Koszul for certain choices of non-maximal
    building sets $\G$, including the example of 
    $A(\L_\M,\G_{\min})$ for $\M=U_{3,4}$ mentioned in
    Remark~\ref{PF-needs-max-building-set}.
\end{remark}

\begin{defn} \rm
Let $\kk$ be a field, and $A$ a {\it finitely generated standard graded associative $\kk$-algebra}.  This means $A$ is a quotient $A=T/I$ where $T=\kk\langle x_1,\ldots,x_n\rangle$ is the free associative algebra on $n$ noncommuting variables $x_1,\ldots,x_n$, considered to all have
$\deg(x_i)=1$, and $I$ is a homogeneous two-sided ideal in $T$.

Writing\footnote{Apologies to the reader that we are writing subscripted $A_k$ here, not superscripted $A^k$ as we did for the Chow rings.} $A=\bigoplus_{k=0}^\infty A_k$, let 
$A_+:=\bigoplus_{k=1}^\infty A_k$ be the maximal graded two-sided ideal of $A$.  Regard the field $\kk$ as an $A$-module via the quotient surjection $A \twoheadrightarrow A/A_+ \cong \kk$.  In other words,
each $x_i$ acts as $0$ on $\kk$.  

Say that $A$ is a {\it Koszul algebra} if the above
surjection $A \twoheadrightarrow \kk$ extends to a {\it linear} graded free $A$-resolution of $\kk$ as an $A$-module, meaning that the $i^{th}$ resolvent $F_i = A(-i)^{\beta_i}$ for some $\beta_i \geq 0$:
$$
0 \leftarrow \kk 
\leftarrow A
\leftarrow A(-1)^{\beta_1} 
\leftarrow A(-2)^{\beta_2} 
\leftarrow A(-3)^{\beta_3} 
\leftarrow \cdots
$$
\end{defn}
There are several equivalent ways to say when $A$ is Koszul, such
as requiring that the polynomial grading of $\Tor^A_i(\kk,\kk)$ is concentrated in degree $i$.  Equivalently, this means that if one starts with the {\it bar complex} $\mathcal{B}_A$ as an $A$-resolution of $\kk$, and then tensors over $A$ with $\kk$, one obtains a complex $\kk \otimes_A \mathcal{B}_A$ of graded $\kk$-vector spaces whose $i^{th}$ homology is concentrated in degree $i$.  The latter characterization leads
to the following result of Polishchuk and Positselski.

\begin{thm} \cite[Chap. 2, Prop. 8.3]{PP}
For any Koszul algebra $A$, and any composition $(\alpha_1,\ldots,\alpha_\ell)$ of $m=\sum_i \alpha_i$, there exists a subcomplex $(C_*,d)$ of $\kk \otimes_A \mathcal{B}_A$ of the form
$
0\rightarrow C_\ell \rightarrow C_{\ell-1} \rightarrow \cdots \rightarrow 
C_1 \rightarrow 0
$
starting with 
$C_\ell=A_{\alpha_1} \otimes \cdots \otimes A_{\alpha_\ell}$ at left, ending with
$C_1=A_{\alpha_1+\cdots+\alpha_\ell}=A_m$ at right,
and with $i^{th}$ term 
$$
C_i=\bigoplus_\beta 
A_{\beta_1} \otimes \cdots \otimes A_{\beta_i}
$$
where $\beta$ in the direct sum runs over all compositions
with $i$ parts that coarsen $\alpha$.
This complex $(C_*,d)$ is exact except at the left end $C_\ell$, meaning that this complex is exact:
\begin{equation}
\label{PP-exact=sequence}
0 \rightarrow \ker(d_\ell) \rightarrow C_\ell \rightarrow C_{\ell-1} \rightarrow \cdots \rightarrow 
C_1 \rightarrow 0
\end{equation}
The complex  $(C_*,d)$ is also
$\subgroup$-equivariant for any group $\subgroup$ of graded $\kk$-algebra automorphisms of $A$.
\end{thm}

Taking the alternating sum of the Euler characteristics term-by-term in \eqref{PP-exact=sequence} yields the
following, where here we conflate a $\kk \subgroup$-module $A_k$
with its character $\chi_{A_k}$.

\begin{cor}
\label{corollary-of-P-P-exact-sequence}
In the above setting, the character of the $\kk \subgroup$-module
$\ker(d_\ell: C_\ell \rightarrow C_{\ell-1})$
has this expression
\begin{align*}
\chi_{\ker(d_\ell)}
=\sum_{i=1}^\ell (-1)^{\ell-i} \chi_{C_i}
&=\sum_{i=1}^\ell (-1)^{\ell-i}
 \sum_{\substack{\beta \colon \ell(\beta)=i \\ \beta \: {\rm      coarsens   } \: \alpha}} 
  A_{\beta_1} \otimes \cdots \otimes A_{\beta_i}\\ 
&=\det\left[
\begin{matrix}
A_{\alpha_1} & A_{\alpha_1+\alpha_2}  & A_{\alpha_1+\alpha_2+\alpha_3}  & \cdots &A_m  \\
A_0 & A_{\alpha_2} & A_{\alpha_2+\alpha_3}  &\cdots  &A_{m-\alpha_1}    \\
0 & A_0 & A_{\alpha_3} & \cdots & A_{m-(\alpha_1+\alpha_2)}   \\
0 & 0 &   &   & \vdots\\
\vdots & \vdots &  &  & A_{\alpha_{\ell-1}+\alpha_\ell}\\
0 & 0   &   \cdots     & A_0 &A_{\alpha_\ell}
\end{matrix}
\right]
\end{align*}
as an $\ell \times \ell$ Toeplitz matrix minor for the sequence of $\kk \subgroup$-modules $(A_0,A_1,A_2,\ldots)$.  In particular, when $\kk=\C$,
then all Toeplitz minors
of this form are genuine characters in $R_\C(\subgroup)$.
\end{cor}

\begin{example}\rm
When $\ell=2$ so that $\alpha=(j,k)$, the exact sequence \eqref{PP-exact=sequence} looks like
$$
0 \rightarrow \ker(d_2)
\rightarrow A_j \otimes A_k 
\rightarrow A_{j+k}
\rightarrow 0
$$
giving this character identity 
$$
\det\left[
\begin{matrix}
A_j & A_{j+k} \\
A_0 & A_k
\end{matrix}
\right]
=\chi_{\ker(d_2)} \quad (\geq_{R_\C(\subgroup)} 0 \text{ if }\kk=\C).
$$

When $\ell=3$ so that $\alpha=(a,b,c)$ , the exact sequence \eqref{PP-exact=sequence} looks like
$$
0 \rightarrow \ker(d_3)
\rightarrow A_a \otimes A_b \otimes A_c 
\rightarrow 
\begin{matrix}
    A_{a+b} \otimes A_c \\
\oplus \\
A_a \otimes A_{b+c}
\end{matrix}
\rightarrow A_{a+b+c}
\rightarrow 0
$$
giving this character identity 
$$
\det\left[
\begin{matrix}
A_a & A_{a+b} & A_{a+b+c}\\
A_0 & A_b & A_{b+c}\\
0 & A_0 & A_{c}
\end{matrix}
\right]
=\chi_{\ker(d_3)} \quad (\geq_{R_\C(\subgroup)} 0 \text{ if }\kk=\C).
$$

When $\ell=4$ so that $\alpha=(a,b,c,d)$, the exact sequence \eqref{PP-exact=sequence} looks like 
$$
0 \rightarrow \ker(d_4)
\rightarrow A_a \otimes A_b \otimes A_c  \otimes A_d
\rightarrow 
\begin{matrix}
A_{a+b} \otimes A_c \otimes A_d\\ 
\oplus \\
A_a \otimes A_{b+c} \otimes A_d \\
\oplus\\
A_a \otimes A_b \otimes A_{c+d}
\end{matrix}
\rightarrow 
\begin{matrix}
A_{a+b+c} \otimes A_d \\
\oplus\\
A_{a+b} \otimes A_{c+d}\\
\oplus\\
A_a \otimes A_{b+c+d}
\end{matrix}
\rightarrow A_{a+b+c+d}
\rightarrow 0
$$
giving this character identity 
$$
\det\left[
\begin{matrix}
A_a & A_{a+b} & A_{a+b+c}&A_{a+b+c+d}\\
A_0 & A_b & A_{b+c}&A_{b+c+d}\\
0 & A_0 & A_{c}& A_{c+d}\\
0 & 0 & A_0 & A_d
\end{matrix}
\right]
=\chi_{\ker(d_4)} \quad (\geq_{R_\C(\subgroup)} 0 \text{ if }\kk=\C).
$$    
\end{example}

One can apply this to Chow rings of
matroids using work of Maestroni and McCullough \cite{MaestroniMcCullough}.

\begin{thm} \cite{MaestroniMcCullough}
For simple matroids $\M$, the Chow ring $A(\L_\M,\G_{\max})$ is Koszul.
\end{thm}

This gives the following promised generalization of 
\eqref{multiplication-surjects}.

\begin{cor}
\label{Chow-ring-Koszul-consequence}
For a matroid $\M$ of rank $r+1$ with
Chow ring $A(\L_\M,\G_{\max})=\bigoplus_{k=0}^r A_k$, and any composition $\alpha=(\alpha_1,\ldots,\alpha_\ell)$ with $m:=\sum_i \alpha \leq r$, the $\ell \times \ell$ Toeplitz minor determinant as shown in
in Corollary~\ref{corollary-of-P-P-exact-sequence} is a genuine character in $R_\C(\subgroup)$ for $\subgroup=\Aut(\M)$.
\end{cor}

Here is the conjectural lift of the previous corollary to Burnside rings, whose $2\times 2$-case is
Proposition~\ref{thm:2by2kos}.

\begin{conj}
\label{Burnside-Koszul-nonnegativity}
In the same context as Corollary~\ref{Chow-ring-Koszul-consequence}, the analogous Toeplitz minors of $\subgroup$-sets have
$$
\det\left[
\begin{matrix}
[\FY^{\alpha_1}] & [\FY^{\alpha_1+\alpha_2}]  & [\FY^{\alpha_1+\alpha_2+\alpha_3}]  & \cdots &[\FY^m]  \\
[\FY^0] & [\FY^{\alpha_2}] & [\FY^{\alpha_2+\alpha_3}]  &\cdots  &[\FY^{m-\alpha_1}]    \\
0 & [\FY^0] & [\FY^{\alpha_3}] & \cdots & [\FY^{m-(\alpha_1+\alpha_2)}]   \\
0 & 0 &   &   & \vdots\\
\vdots & \vdots & & & [\FY^{\alpha_{\ell-1}+\alpha_\ell}]\\
0 & 0   &   \cdots     & [\FY^0] &[\FY^{\alpha_\ell}]
\end{matrix}
\right] \geq_{B(\subgroup)} 0.
$$
\end{conj}

As a bit of further evidence for Conjecture~\ref{Burnside-Koszul-nonnegativity}, we check it for the Toeplitz
minors with $\alpha=(1,1,1)$.

\begin{thm} \label{thm:3by3}
For any matroid $\M$, the Chow ring $A(\L_\M,\G_{\max})$ has 
    \begin{align*}
    \det\left|\begin{matrix}
    [\FY^1] & [\FY^2] & [\FY^3] \\
    [\FY^0] & [\FY^1] & [\FY^2] \\
    0 & [\FY^0] & [\FY^1]
    \end{matrix}\right| \geq_{B(\subgroup)} 0.
    \end{align*}
\end{thm}

\begin{proof}
Multiplying out the determinant, one needs to 
prove the following inequality in $B(\subgroup)$:
$$
[\FY^1 \times \FY^1 
    \times \FY^1] - \left( \begin{matrix} [\FY^2 \times \FY^1]\\ +\\ [\FY^1 \times \FY^2]\end{matrix} \right)  + [\FY^3] \,\, \geq_{B(\subgroup)} 0,
$$
or equivalently, one must show the inequality
$$
[\quad (\FY^2 \times \FY^1)  \,\, \sqcup \,\, (\FY^1 \times \FY^2)\quad ]
\,\, \leq_{B(\subgroup)} \,\, 
[\quad (\FY^1 \times \FY^1 
    \times \FY^1) \,\, \sqcup \,\, \FY^3\quad ].
$$
For this, it suffices to provide an injective $\subgroup$-equivariant map 
\[
(\FY^1 \times \FY^2) \sqcup (\FY^2 \times \FY^1)
\,\, \into \,\, 
(\FY^1 \times \FY^1 
    \times \FY^1) \sqcup \FY^3.
\] 
Such a map is summarized schematically in Figures~\ref{fig: 3x3a} and \ref{fig: 3x3b}, with certain abbreviation conventions: the variables $x,y,z$ always abbreviate
the variables $x_{F_1}, x_{F_2},x_{F_3}$ for a generic nested flag of flats $F_1 \subset F_2 \subset F_3$, while the variable $w$ abbreviates $x_F$ for a flat $F$ incomparable to any of $F_1,F_2,F_3$.

The Figures \ref{fig: 3x3a} and \ref{fig: 3x3b} 
describe for each type of element in $\FY^1 \times \FY^2$ and $\FY^2 \times \FY^1$ an appropriate image in either $\FY^1 \times \FY^1 \times \FY^1$ or $\FY^3$.
Loosely speaking, the maps try to send elements to $\FY^3$ whenever possible, that is, whenever their product is a valid element of $\FY^3$. When this fails, we find an image in $\FY^1 \times \FY^1 \times \FY^1$, carefully trying to keep track of which images have been used by noting various conditions on the $x,y,z,$ and $w$, involving their ranks  and sometimes their {\it coranks}, denoted $\cork(F):=\rk(E)-\rk(F)$. Conditions in gray are forced by the form of the given tuple, while conditions in black are assumed to separate the map into disjoint cases.$\qedhere$

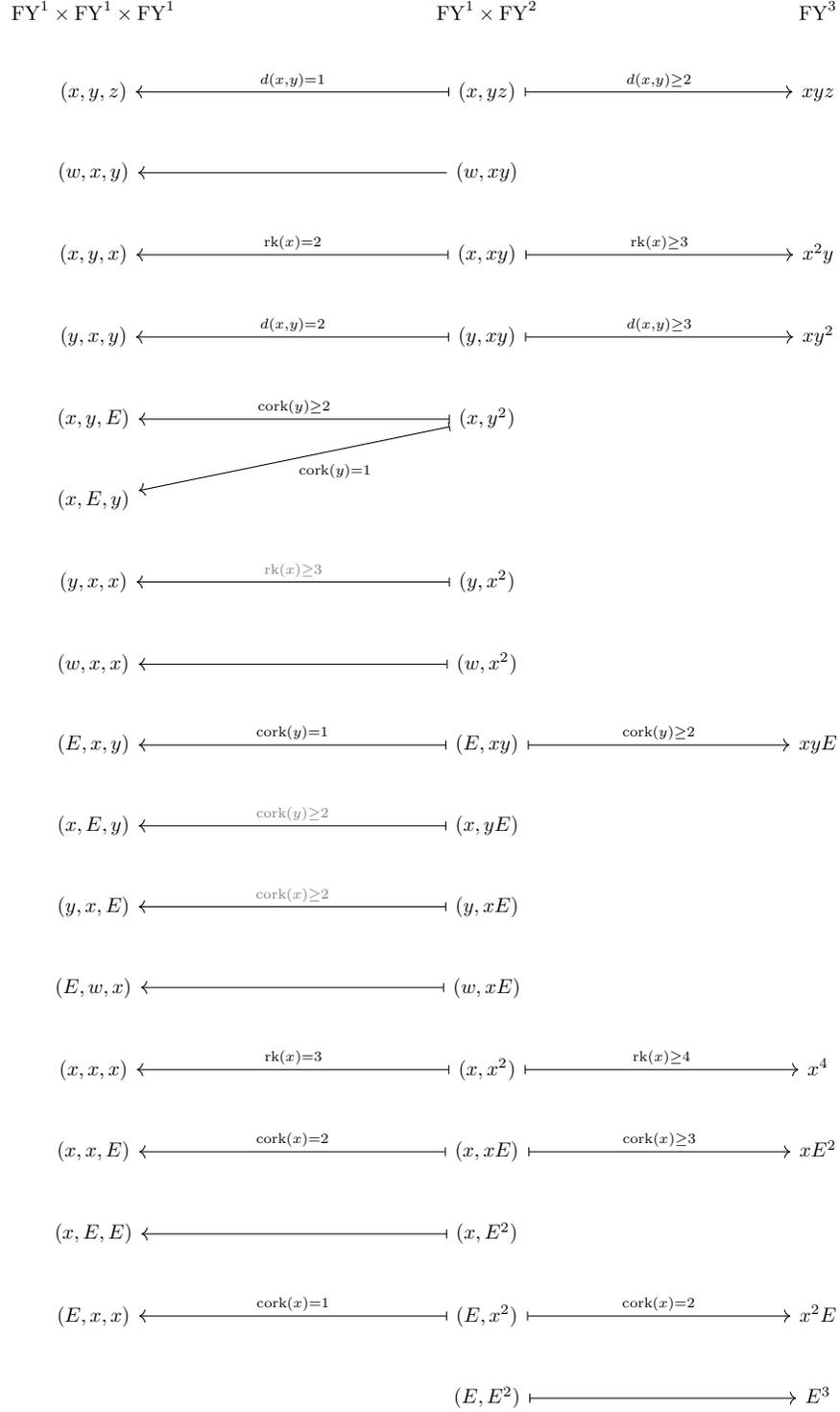
\begin{figure}[ht] 
    \centering
\begin{tikzcd}[scale cd=0.8]
	{\FY^1\times \FY^1\times \FY^1} &&&& {\FY^1\times \FY^2} &&&& {\FY^3} 
    \\
	{(x, y, z)} &&&& {(x, yz)} &&&& xyz
	\arrow["{d(x, y) =1}"', maps to, from=2-5, to=2-1] 
 	\arrow["{d(x, y) \geq 2}", maps to, from=2-5, to=2-9]
    \\
    {(w, x, y)} &&&& {(w, xy)} &&&& 
	\arrow[from=3-5, to=3-1]
     \\
     {(x, y, x)} &&&& {(x, xy)} &&&& x^2y
	\arrow["{\rk(x) = 2}"', maps to, from=4-5, to=4-1] 
 	\arrow["{\rk(x)\geq 3}", maps to, from=4-5, to=4-9]
    \\
    {(y, x, y)} &&&& {(y, xy)} &&&& xy^2
	\arrow["{d(x, y) = 2}"', maps to, from=5-5, to=5-1] 
 	\arrow["{d(x, y) \geq 3}", maps to, from=5-5, to=5-9]
    \\
    {(x, y, \E)} &&&& {(x, y^2)} &&&& 
	\arrow["{\cork(y) \geq 2}"', maps to, from=6-5, to=6-1] 
    \\
    {(x, \E, y)} &&&&  &&&&
	\arrow["{\cork(y) = 1}", maps to, from=6-5, to=7-1] 
    \\
    {(y, x, x)} &&&& {(y, x^2)} &&&& 
	\arrow["{\color{gray} \rk(x) \geq 3}"', maps to, from=8-5, to=8-1] 
    \\
    {(w, x, x)} &&&& {(w, x^2)} &&&& 
    \arrow[maps to, from=9-5, to=9-1]
    \\
    {(\E, x, y)} &&&& {(\E, xy)} &&&& xy\E
    \arrow["{\cork(y) = 1}"', maps to, from=10-5, to=10-1]
    \arrow["{\cork(y) \geq 2}", maps to, from=10-5, to=10-9]
    \\
    {(x, \E, y)} &&&& {(x, yE)} &&&& 
    \arrow["{\color{gray} \cork(y) \geq 2}"', maps to, from=11-5, to=11-1]
    \\
    {(y, x, \E)} &&&& {(y, xE)} &&&& 
    \arrow["{\color{gray} {\cork(x) \geq 2}}"', maps to, from=12-5, to=12-1]
    \\
    {(\E, w, x)} &&&& {(w, xE)} &&&& 
    \arrow[maps to, from=13-5, to=13-1]
    \\
    {(x, x, x)} &&&& {(x, x^2)} &&&& x^4
    \arrow["{\rk(x) = 3}"', maps to, from=14-5, to=14-1]
    \arrow["{\rk(x) \geq 4}", maps to, from=14-5, to=14-9]
    \\
     {(x, x, \E)} &&&& {(x, xE)} &&&& xE^2
    \arrow["{\cork(x) = 2}"', maps to, from=15-5, to=15-1]
    \arrow["{\cork(x) \geq 3}", maps to, from=15-5, to=15-9]
    \\
    {(x, \E, \E)} &&&& {(x, E^2)} &&&&
    \arrow[maps to, from=16-5, to=16-1]
    \\
    {(\E, x, x)} &&&& {(E, x^2)} &&&& x^2E
    \arrow["{\cork(x) = 1}"', maps to, from=17-5, to=17-1]
    \arrow["{\cork(x) = 2}", maps to, from=17-5, to=17-9]
    \\
    &&&& {(\E, \E^2)} &&&& \E^3
    \arrow[maps to, from=18-5, to=18-9]
    \\ \\
\end{tikzcd}
\caption{Part of the injective set map $\FY^1\times \FY^2\into \FY^1\times \FY^1\times \FY^1 \cup \FY^3$}
\label{fig: 3x3a}
\end{figure}

\begin{figure}[ht] 
    \centering
\begin{tikzcd}[scale cd=0.5]
	{\FY^1\times \FY^1\times \FY^1} &&&& {\FY^2\times \FY^1} &&&& {\FY^3} 
    \\
	{(y,z, x)} &&&& {(yz, x)} &&&& 
	\arrow["{\color{gray}{d(y, z) \geq 2} }"', maps to, from=2-5, to=2-1] 
    \\
    {(x, y, w)} &&&& {(xy, w)} &&&& 
	\arrow[from=3-5, to=3-1]
     \\
     {(x, y, x)} &&&& {(xy, x)} &&&& 
	\arrow["{\rk(x)\geq 3}"', maps to, from=4-5, to=4-1] 
    \\
    {(x, x, y)} &&&& &&&& 
	\arrow["{\rk(x) = 2}", maps to, from=4-5, to=5-1] 
    \\
    {(y, x, y)} &&&& {(xy, y)} &&&& 
	\arrow["{d(x, y) \geq 3}"', maps to, from=6-5, to=6-1] 
    \\
    {(x, y, y)} &&&& &&&& 
	\arrow["{d(x, y) = 2}", maps to, from=6-5, to=7-1] 
    \\
    {(y, y, x)} &&&& {(y^2, x)} &&&& 
	\arrow["{\color{gray} \rk(y) \geq 3}"', maps to, from=8-5, to=8-1] 
    \\
    {(x, x, y)} &&&& {(x^2, y)} &&&& 
	\arrow["{\color{gray} \rk(x) \geq 3}"', maps to, from=9-5, to=9-1] 
    \\
    {( x, x, w)} &&&& {(x^2, w)} &&&& 
    \arrow[maps to, from=10-5, to=10-1]
    \\
    {(\E, y, x)} &&&& {(xy, \E)} &&&& 
    \arrow["{\cork(y) \geq 2}"', maps to, from=11-5, to=11-1]
    \\
    {(x,  y, \E)} &&&&  &&&& 
    \arrow["{\cork(y) = 1}", maps to, from=11-5, to=12-1]
    \\
    {(y, \E, x)} &&&& {(y\E, x)} &&&& 
    \arrow["{\color{gray} {\cork(y) \geq 2}}"', maps to, from=13-5, to=13-1]
    \\
    {(\E, x, y)} &&&& {(x\E, y)} &&&& 
    \arrow["{\color{gray} {\cork(x) \geq 2}}"', maps to, from=14-5, to=14-1]
    \\
    {(\E, x, w)} &&&& {(xE, w)} &&&& 
    \arrow[maps to, from=15-5, to=15-1]
    \\
    {(x, x, x)} &&&& {(x^2, x)} &&&& 
    \arrow["{\rk(x) \geq 4}"', maps to, from=16-5, to=16-1]
    \\
    {(E, x, E)} &&&&  &&&& 
    \arrow["{\rk(x) = 3}", maps to, from=16-5, to=17-1]
    \\
     {(x, x, \E)} &&&& {(xE, x)} &&&& 
    \arrow["{\cork(x) \geq 3}"', maps to, from=18-5, to=18-1]
    \\
    {(x, \E, x)} &&&&  &&&&
    \arrow["{\cork(x) = 2}", maps to, from=18-5, to=19-1]
    \\
    {(\E, \E, x)} &&&& {(E^2, x)} &&&& 
    \arrow[ maps to, from=20-5, to=20-1]
    \\
    {(x, x, \E)}&&&& {(x^2, E)} &&&& 
    \arrow["{\cork(x) = 1}"', maps to, from=21-5, to=21-1]
    \\
    {(\E, x, x)}&&&&  &&&& 
    \arrow["{\cork(x) \geq 2}", maps to, from=21-5, to=22-1]
    \\
    {(\E, \E, \E)}&&&& {(E^2, E)} &&&& 
    \arrow[ maps to, from=23-5, to=23-1]
    \\
\end{tikzcd}
    \caption{The rest of the injective set map $\FY^2\times \FY^1\into \FY^1\times \FY^1\times \FY^1 \cup \FY^3$}
    \label{fig: 3x3b}
\end{figure}
\end{proof}

\section{Further questions and conjectures}
\label{sec: further-questions}
In addition to Conjectures~\ref{log-concavity-conjectures}, \ref{boolean-h-log-concavity-conj}, \ref{Burnside-Koszul-nonnegativity} above,
we collect here are a few more questions and conjectures.

\subsection{Explicit formulas for Chow rings as permutation representations?}
In \cite[Lem. 3.1]{Stembridge}, Stembridge provides a generating function for the symmetric group
representations on each graded component of the Chow ring for all Boolean matroids; see also Liao \cite{Liao, Liao_new}.  Furthermore, Stembridge's expression exhibits them as {\it permutation representations}, whose orbit-stabilizers are all {\it Young subgroups} in the symmetric group.
This prompts a slightly vague question.
\begin{question}
Can one provide similarly explicit generating function expressions as permutation representations for other families of
matroids with symmetry?
\end{question}

\subsection{Equivariant $\gamma$-positivity?}
Hilbert functions $(a_0,a_1,\ldots,a_r)$ for
Chow rings $A(\L_\M,\G_{\max})$ of rank $r+1$ simple matroids are
not only symmetric and unimodal, but satisfy the stronger condition of
{\it $\gamma$-positivity}:  one has {\it nonnegativity} for all coefficients 
$\gamma=(\gamma_0,\gamma_1,\ldots,\gamma_{\lfloor \frac{r}{2} \rfloor})$ appearing in the unique expansion
\begin{equation}
\label{gamma-defining-relation}
\sum_{i=0}^r a_i t^i = \sum_{i=0}^{\lfloor\frac{r}{2}\rfloor} \gamma_i \,\, t^i(1+t)^{r-2i}.
\end{equation}
See Athanasiadis \cite{Athanasiadis} for a nice survey on $\gamma$-positivity.
It has been shown, independently 
by Ferroni, Matherne, Stevens and Vecchi \cite[Thm. 3.25]{FMSV}
and by Wang (see \cite[p.~29]{FMSV}), that the $\gamma$-positivity for Hilbert series of Chow rings of matroids
follows from results of Braden, Huh, Matherne, Proudfoot and Wang \cite{BHMPW} on {\it semismall decompositions}.
The maximal building set is again important: the
same example of $A(\L_\M,\G_{\min})$ for $\M=U_{3,4}$ from Remark~\ref{PF-needs-max-building-set} with $(a_0,a_1,a_2)=(1,1,1)$ has $\gamma=(\gamma_0,\gamma_1)=(1,-1)$.

One also has the notion of {\it equivariant $\gamma$-positivity} for a sequence
of $\gamma$-representations $(A_0,A_1,\ldots,A_r)$, 
due originally to Shareshian and Wachs
\cite[\S 5]{ShareshianWachs} (see also \cite[\S5.2]{Athanasiadis}, \cite[Def. 4.13]{FMSV}): upon replacing each
$a_i$ in \eqref{gamma-defining-relation} with the element $[A_i]$ of $R_\C(\gamma)$, one
asks that the uniquely defined coefficients $\gamma_i$ in $R_\C(\gamma)$ have $\gamma_i \geq_{R_\C(\gamma)} 0$.
Computations suggested the next conjecture, which appeared in the first {\tt arXiv} version of this paper, and which was then proven by Hsin-Chieh Liao\footnote{Personal communication, February 2024; see also \cite[\S8]{liao-q-uniform}}.

\begin{conj}
\label{equivariant-gamma-positivity-conj}
For any matroid $\M$ of rank $r+1$ and its Chow ring
$A(\L_\M,\G_{\max})=\bigoplus_i A^i$,
the sequence of $\gamma$-representations
$(A^0_\C,A^1_\C,\ldots,A^r_\C)$ is equivariantly $\gamma$-positive.
\end{conj}
\noindent
For example, \cite[Cor. 5.4]{ShareshianWachs} verifies Conjecture~\ref{equivariant-gamma-positivity-conj} for Boolean matroids.
However, one can check that the stronger conjecture of {\it Burnside $\gamma$-nonnegativity} for $(\FY^0,\FY^1,\ldots,\FY^r)$
would  {\it fail} already for the Boolean 
matroid of rank $3$:  here
$\FY^0, \FY^2$ carry the trivial $\symm_3$ permutation representation $\mathbf{1}$ , while $\FY^1$ carries the defining $\symm_3$-permutation representation on the set $X=\{1,2,3\}$, so
$\gamma_0=[\mathbf{1}]$, but $\gamma_1=[X]-[\mathbf{1}] \not\geq_{B(\symm_3)} 0$.

\printbibliography
\end{document}